\documentclass[a4paper,fleqn]{cas-sc}

\usepackage[numbers]{natbib}
\usepackage{graphicx}%
\usepackage[export]{adjustbox}
\usepackage{multirow}%
\usepackage{amsmath,amssymb,amsfonts}%
\usepackage{amsthm}%
\usepackage{mathtools}
\usepackage{mathrsfs}%
\usepackage[title]{appendix}%
\usepackage{xcolor}%
\usepackage{textcomp}%
\usepackage{manyfoot}%
\usepackage{booktabs}%
\usepackage{algorithm}%
\usepackage{algorithmicx}%
\usepackage{algpseudocode}%
\usepackage{listings}%
\usepackage{soul}%

\usepackage{bm}
\usepackage{enumerate}
\usepackage[utf8]{inputenc}
\usepackage{mparhack}
\usepackage{marginnote}
\usepackage{anyfontsize}

\newcommand{\setdef}[2]{\left\{\ #1\ \left|\ \vphantom{#1} #2\ \right.\right\}}

\DeclareMathOperator{\curl}{curl}
\DeclareMathOperator{\dom}{dom}
\DeclareMathOperator{\divg}{div}
\DeclareMathOperator{\im}{im}
\DeclareMathOperator{\id}{id}

\DeclareMathOperator{\supp}{supp}

\newcommand{\R}{\mathbb{R}}

\newcommand{\ext}{\mathrm{ext}}
\newcommand{\natNum}{\mathbb{N}}

\newcommand{\charFunction}[1]{\mathbb{1}_{#1}}

\theoremstyle{plain}
\newtheorem{theorem}{Theorem}
\newtheorem{proposition}[theorem]{Proposition}
\newtheorem{lemma}[theorem]{Lemma}

\theoremstyle{definition}
\newtheorem{definition}[theorem]{Definition}

\theoremstyle{remark}
\newtheorem{remark}[theorem]{Remark}%%%

\begin{document}
\let\WriteBookmarks\relax
\def\floatpagepagefraction{1}
\def\textpagefraction{.001}

% Short title
\shorttitle{Analysis of an eddy current brake model}    

% Short author
\shortauthors{A.~Fischer, K.~Krha\v{c}, T.~Reis, A.A.~Wierzba}  

% Main title of the paper
\title[mode=title]{Analysis of a coupled magneto-quasistatic--mechanical model of an eddy current brake}  

% Title footnote mark
\tnotemark[1]

% Title footnote 1.
\tnotetext[1]{This work was funded by the Deutsche Forschungsgemeinschaft (DFG, German
Research Foundation) through the Collaborative Research Centre CRC~1701
``Port-Hamiltonian Systems'', project number 531152215.}

% First author
%
% Options: Use if required
% eg: \author[1,3]{Author Name}[type=editor,
%       style=chinese,
%       auid=000,
%       bioid=1,
%       prefix=Sir,
%       orcid=0000-0000-0000-0000,
%       facebook=<facebook id>,
%       twitter=<twitter id>,
%       linkedin=<linkedin id>,
%       gplus=<gplus id>]

\author[1]{Anna Fischer}[orcid=0009-0002-4855-438X]
\credit{Conceptualization, Methodology, Formal analysis,
Writing -- original draft, Writing -- review \& editing}
\ead{anna.fischer@tu-ilmenau.de}

\author[2]{Kaja Krha\v{c}}
\credit{Conceptualization, Methodology, Formal analysis,
Writing -- original draft, Writing -- review \& editing}
\ead{krhac@uni-wuppertal.de}

\author[1]{Timo Reis}[orcid=0000-0003-0721-8494]
\cormark[1]
\credit{Conceptualization, Methodology, Formal analysis,
Writing -- original draft, Writing -- review \& editing}
\ead{timo.reis@tu-ilmenau.de}

\author[2]{Alexander A. Wierzba}[orcid=0000-0002-2013-0246]
\credit{Conceptualization, Methodology, Formal analysis,
Writing -- original draft, Writing -- review \& editing}
\ead{alexander.wierzba@uni-wuppertal.de}

% Address/affiliation
\affiliation[1]{organization={Institut f\"ur Mathematik, Technische Universit\"at Ilmenau}, 
	addressline={Weimarer Stra{\ss}e~25}, 
	          citysep={},
		postcode={98693}, 
		city={Ilmenau}, 
		country={Germany}}

\affiliation[2]{organization={Fakult\"at f\"ur Mathematik und Naturwissenschaften, Bergische Universit\"at Wuppertal}, 
  addressline={Gau{\ss}stra{\ss}e 20}, 
    postcode={42119},
    city={Wuppertal}, 
    citysep={},
    country={Germany}}

% Corresponding author text
\cortext[1]{Corresponding author}

% For a title note without a number/mark
%\nonumnote{}

% Here goes the abstract
\begin{abstract}
We study a coupled magneto-quasistatic model for an eddy current
brake. The model leads to a nonlinear infinite-dimensional
differential-algebraic system in which the nonlinearity is structural rather
than caused by nonlinear material laws. It is generated by the closed
electromechanical feedback loop: the angular momentum determines the velocity
of the conducting disk, the velocity and the magnetic field generate motional
eddy currents, and these currents induce a Lorentz torque acting back on the
rotor. We introduce a finite-dissipation solution concept and prove global
existence and uniqueness of weak solutions for initial data given by the
magnetic field and the angular momentum, and for inputs given by the supplied
coil current and the externally applied torque.
\end{abstract}

% Keywords
% Each keyword is seperated by \sep
\begin{keywords}
magneto-quasistatic system \sep
Maxwell's equations \sep
nonlinear evolution equation \sep
differential-algebraic equation \sep
electromechanical coupling \sep
eddy current brake
\MSC[2020]{34A09, 37L05, 78A30, 35K90, 47J35}
\end{keywords}

\maketitle

% Main text

\section{Introduction}    
Eddy current brakes use magnetic forces to decelerate a moving conductor. Since their operation is contactless, and thus without friction-induced wear, they are particularly attractive for applications such as high-speed trains, amusement rides, industrial machinery, and fitness equipment \cite{Ramharack2023,Sivakumar2014,US5656001A}. Mathematical models and numerical simulations of such devices have been considered, for instance, in \cite{Karakoc2016,Fan2024}. From a modeling perspective, an eddy current brake can be described by a coupled PDE--ODE system: the electromagnetic field is governed by Maxwell's equations, while the rotational motion of the disk is represented by an ordinary differential equation. The coupling between both parts is induced by the Lorentz force.

In the configuration considered here, displacement currents are negligible
compared with conduction currents. It is therefore natural to replace the full
Maxwell system by the magneto-quasistatic approximation, obtained by
neglecting the displacement current term in Ampere's law. This approximation
leads to algebraic constraints in the non-conducting subdomains, such as the
air gap between the electromagnet and the rotating disk. As a consequence, the
overall model naturally takes the form of a partial differential-algebraic
equation (PDAE).

The resulting system is nonlinear even for linear conductivity and reluctivity
laws. This nonlinearity is not due to nonlinear material parameters, but to the
structure of the electromechanical coupling. The angular momentum determines
the disk velocity, the disk velocity produces a motional electric field
$\bm v\times\bm B$, and the induced current generates a Lorentz torque acting
back on the angular momentum. Thus the electromagnetic and mechanical
subsystems form a nonlinear feedback loop.
%For such systems, the literature currently lacks a comprehensive, general solution theory.

In this paper, we analyse this nonlinear coupled system. The key point is to
separate the closed linear magneto-quasistatic realization from the nonlinear
electromechanical feedback. The magnetic subsystem is realized by a closed
nonnegative form on a finite-dissipation state space, while the motional
current and the Lorentz torque are treated as nonlinear coupling terms. This
leads to an integral formulation in which the nonlinear feedback is handled by
a Picard fixed-point argument. The continuation to arbitrary finite time
intervals is obtained from the energy balance, where the Lorentz coupling
terms cancel. We also record the formal port-Hamiltonian energy structure of
the model.

Most mathematical results for eddy current models focus on the
electromagnetic subsystem without coupling to the mechanical dynamics
\cite{PaulyPicard2017,PaulyPicardTrostorffWaurick2021,ChillReisStykel2023}.
Such models typically lead to mixed elliptic--parabolic systems, and
nonlinearities usually enter through material laws or constitutive relations.
The eddy current brake considered here is different: even with linear
materials, the coupling to the rotating conductor produces a nonlinear
feedback between the magnetic field and the mechanical angular momentum. To
the best of our knowledge, an existence and uniqueness analysis for this
fully coupled magneto-quasistatic--mechanical eddy current brake model has not
yet been available in the literature.

This paper is organized as follows. Section~\ref{sec:modeling} derives the eddy current brake model from the magneto-quasistatic Maxwell equations and the mechanical torque balance. Section~\ref{sec:solution_concept} collects the standing hypotheses, introduces the weak solution concept, and records the formal port-Hamiltonian energy structure of the model. In Section~\ref{sec:mqs}, we construct the linear magneto-quasistatic realization, including its inhomogeneous form formulation in the extrapolation space. Section~\ref{sec:solvability_analysis} proves existence and uniqueness for the coupled electromagnetic--mechanical system by a Picard fixed-point argument. Section~\ref{sec:model_extensions} discusses possible model extensions, including the full Maxwell system and non-axisymmetric disks.

\section{Model derivation}\label{sec:modeling}

This section derives the formal model used in the subsequent analysis. We first recall the magneto-quasistatic Maxwell equations and the constitutive relations. We then describe the rotating disk, the Lorentz-force coupling, and the resulting partial differential-algebraic system. The precise assumptions and the weak interpretation of the equations are postponed to Section~\ref{sec:solution_concept}.

\subsection{Eddy current brake model}\label{sec:model}

%The model that forms the basis for the detailed analysis carried out in the subsequent sections. 
The electromagnetic field variables are the $\R^3$-valued functions
\[
\begin{aligned}
\bm{B}&:\ \text{magnetic flux density}, & \bm{H}&:\ \text{magnetic field strength},\\
\bm{E}&:\ \text{electric field strength}, & \bm{J}&:\ \text{electric current density},
\end{aligned}
\]
which are defined on $\Omega \times [0,T)$, where $\Omega \subseteq \R^3$ denotes a bounded computational domain and $T \in (0,\infty]$ denotes the time horizon.
Since displacement currents are negligible in the eddy current regime considered here, it is reasonable to employ the \emph{magneto-quasistatic approximation}, obtained by setting ${\textstyle\frac{\partial}{\partial t}}\bm{D}\equiv 0$ in Amp\`ere's law. 
The magneto-quasistatic Maxwell equations read
\begin{align*}
  \tfrac{\mathrm{d}}{\mathrm{d}t} \bm{B}(\xi,t) &= -\,\curl \bm{E}(\xi,t), \\
  0 &= \phantom{-}\curl \bm{H}(\xi,t) - \bm{J}(\xi,t),
\end{align*}
and are supplemented with a static Silver--Müller-type absorbing boundary condition on the artificial outer boundary
\cite[Sec.~13.5.1]{Monk2003},
\begin{equation}
\label{eq:silver_mueller_bc}
\bm{E}(\xi,t)\times \bm{n}_o(\xi)
=
-\zeta(\xi)\,
\bm{n}_o(\xi)\times
\bigl(\nu(\xi)\bm{B}(\xi,t)\times \bm{n}_o(\xi)\bigr),
  \qquad \xi\in\partial\Omega,\ t\in[0,T),
\end{equation}
where $\bm{n}_o$ denotes the outward unit normal vector on $\partial\Omega$ and
$\zeta$ is a nonnegative scalar boundary absorption coefficient.
For $\zeta=0$, \eqref{eq:silver_mueller_bc} reduces to the homogeneous tangential boundary condition. The precise trace interpretation of \eqref{eq:silver_mueller_bc} will be specified in the solution concept below.

The domain $\Omega$ is understood as a bounded computational domain obtained by truncating the exterior region, and
\eqref{eq:silver_mueller_bc} is imposed on the artificial outer boundary to reduce reflections caused by the truncation of the exterior region.

The model is further equipped with
\emph{constitutive relations} that link $\bm{H}$ to $\bm{B}$ and $\bm{J}$ to $\bm{E}$.
For linear materials, these take the form
\begin{align*}
  \bm{H}(\xi,t) &= \nu(\xi)\,\bm{B}(\xi,t),\\
  \bm{J}(\xi,t) &= \sigma(\xi)\,\bm{E}(\xi,t)
                   + \bm{J}_{\ext}(\xi,t),
\end{align*}
where $\nu,\sigma\colon \Omega \to \R^{3\times 3}$ denote, respectively,
the (tensorial) magnetic reluctivity and electrical conductivity of the material,
and $\bm{J}_{\ext}$ represents externally imposed currents.
As in \cite{ChillReisStykel2023}, we consider a stranded electromagnet as a source of the magnetic field acting on the disk.
In such a device, the injected current is generated by $m$ windings according to
\[
\bm{J}_{\mathrm{em}}(\xi, t) = \chi(\xi) i(t),
\]
where $i$ is an $\mathbb{R}^m$-valued input current and $\chi$ is the $\mathbb{R}^{3 \times m}$-valued winding density describing the geometry of the windings.

\subsection{Rotating disk and Lorentz-force coupling}
\label{subsec:disk_model}

The eddy current brake considered in this paper consists of an electromagnet
and a moving conductor, modeled as a rotating disk of uniform thickness. The disk occupies the domain $\Omega_{\mathrm d}\subset\Omega$, as illustrated in Figure~\ref{fig:disk}.
% \begin{equation}  \label{fig:disk}
% \includegraphics[width=0.5\linewidth, valign=c]{RotatingDisk.png} \quad \,.
% \end{equation}
\begin{figure}
    \centering
    \includegraphics[width=0.75\linewidth]{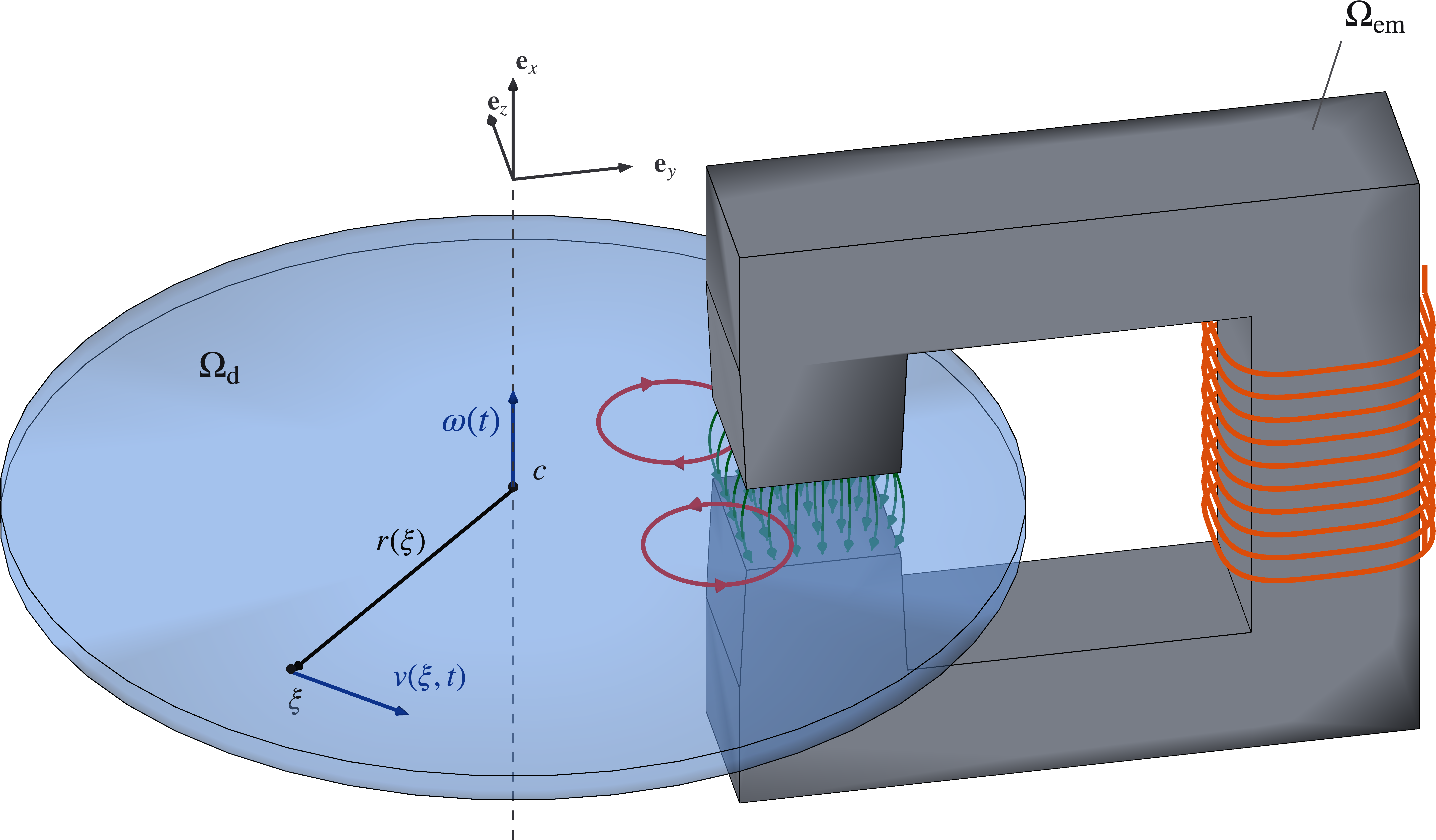}
    \caption{Three-dimensional representation of a rotary eddy current brake and its kinematic and electromagnetic quantities. It consists of a conductive disk with domain $\Omega_{\mathrm{d}}$, which rotates with angular velocity $\omega$ about the $\bm{e}_x$-axis. The electromagnet with domain $\Omega_{\mathrm{em}}$, which includes the excitation windings to the right, produces a magnetic flux density, illustrated by its magnetostatic field lines across the pole-shoe region of the electromagnet. The red closed loops schematically indicate the direction of the eddy currents induced in the disk. }
    \label{fig:disk}
\end{figure}
We assume
that the disk undergoes rigid-body rotation about the fixed axis passing through the center of the disk,
$\{c+s\bm e_x:s\in\R\}$,
where $c\in\Omega_{\mathrm d}$ is a fixed point on the axis of rotation (for instance the
center point of the disk) and
$\bm e_x\in\R^3$ denotes the unit vector in the $x$-direction. For any point $\xi \in \Omega_{\mathrm d}$, we define the position vector relative to the reference point $c$ by
$\bm r(\xi):=\xi-c$.  If
$\omega(t)\in\R$ denotes the angular velocity, then the corresponding velocity field is given by
\[
\bm v(\xi,t)
=
\omega(t)\bigl(\bm e_x\times\bm r(\xi)\bigr),
\qquad \xi\in\Omega_{\mathrm d}.
\]

The eddy currents are induced by the Lorentz force acting on charges moving with velocity $\bm{v}$ in the magnetic field $\bm{B}$. The Lorentz force density per unit charge,  $\bm{v} \times \bm{B}$, has the physical dimension of an electric field. Denoting the conductivity of the disk by $\sigma_{\mathrm d}$, the current density induced by the motion of the disk is
$\bm{J}_{\mathrm{eddy}}=\sigma_{\mathrm{d}}(\bm{v}\times\bm{B})$.
The resulting total current density in the disk is therefore given by
\[
\bm J_{\mathrm d}(\xi,t)
=
\sigma_{\mathrm d}(\xi)
\bigl(
\bm E(\xi,t)+\bm v(\xi,t)\times\bm B(\xi,t)
\bigr),
\qquad \xi\in\Omega_{\mathrm d}.
\]
Here $\bm J_{\mathrm d}$ denotes the total current density in the disk,
including both the Ohmic contribution induced by $\bm E$ and the motional
contribution induced by $\bm v\times\bm B$.

By Lenz's law, the eddy currents $\bm{J}_{\mathrm{eddy}}$ are directed so as to oppose the motion of the conductor. The Lorentz force density acting on the disk is given by $\bm{J}_{\mathrm{d}} \times \bm{B}$, and hence the torque about the axis of rotation is obtained by integrating the corresponding moment density over $\Omega_\mathrm{d}$. If $L(t)$ and $M_{\ext}(t)$ denote, respectively, the angular momentum of the disk and the externally applied torque at time $t$, then the torque balance is given by
\begin{equation*}
    \dot L(t) = \bm{e}_x^\top \int_{\Omega_\mathrm{d}} \bm{r}(\xi) \times \bigl(\bm{J}_{\mathrm{d}}(\xi,t) \times \bm{B}(\xi,t)\bigr)\,\mathrm{d}\xi+M_{\ext}(t).
\end{equation*}
Furthermore, we use the relation $\omega(t)=I_{\mathrm{d}}^{-1}L(t)$, where $I_{\mathrm{d}}$ denotes the moment of inertia of the disk.

\subsection{Coupled partial differential-algebraic model}
\label{subsec:coupled_model}

In addition to the disk domain $\Omega_{\mathrm d}$, we denote by
$\Omega_{\mathrm{em}}$ the conducting part of the electromagnet. The conducting
region is written as
$\Omega_{\mathrm c}:=\Omega_{\mathrm d}\cup\Omega_{\mathrm{em}}$,
and the remaining non-conducting part of the computational domain is denoted
by $\Omega_0$. The precise geometric assumptions on these subdomains are
collected in Subsection~\ref{subsec:standing_assumptions}.

With these conventions, the unknowns of the coupled model are the magnetic
flux density $\bm B$, the angular momentum $L$, the electric field $\bm E$,
and the current density $\bm J_{\mathrm d}$ in the disk. The dynamics of the
coupled electromagnetic--mechanical system are described by
\begin{subequations}
\label{eq:model}
\begin{align}
\tfrac{\mathrm{d}}{\mathrm{d}t} \bm B(\xi,t)
  &=
  -\curl\bm E(\xi,t),
  && \xi\in\Omega,
  \label{eq:dynB}
  \\
\dot L(t)
  &=
  \bm e_x^\top
  \int_{\Omega_{\mathrm d}}
  \bm r(\xi)\times
  \bigl(\bm J_{\mathrm d}(\xi,t)\times\bm B(\xi,t)\bigr)
  \,\mathrm d\xi
  +
  M_{\ext}(t),
  &&
  \label{eq:dynL}
  \\
\bm J_{\mathrm d}(\xi,t)
  &=
  \sigma_{\mathrm d}(\xi)
  \left(
  \bm E(\xi,t)
  +
  \bm v(\xi,t)\times\bm B(\xi,t)
  \right),
  && \xi\in\Omega_{\mathrm d},
  \label{eq:constraint_disk}
  \\
\curl\bigl(\nu\bm B\bigr)(\xi,t)
  &=
  \begin{cases}
  \sigma_{\mathrm{em}}(\xi)\bm E(\xi,t)
  +
  \bm J_{\mathrm{em}}(\xi,t),
  & \xi\in\Omega_{\mathrm{em}},
  \\
  \bm J_{\mathrm d}(\xi,t),
  & \xi\in\Omega_{\mathrm d},
  \\
  0,
  & \xi\in\Omega_0,
  \end{cases}
  \label{eq:constraint_ampere}
  \\
\bm E(\xi,t)\times\bm n_o(\xi)
  &=
  -\zeta(\xi)\,
  \bm n_o(\xi)\times
  \bigl(\nu(\xi)\bm B(\xi,t)\times\bm n_o(\xi)\bigr),
  && \xi\in\partial\Omega.
  \label{eq:model_boundary}
\end{align}
where
\begin{equation}\label{eq:model_abbreviations}
\bm r(\xi):=\xi-c,\qquad
\bm v(\xi,t):=I_{\mathrm d}^{-1}L(t)(\bm e_x\times\bm r(\xi)),
\qquad
\bm J_{\mathrm{em}}(\xi,t):=\chi(\xi)i(t),
\end{equation}
\end{subequations}
with $c\in\Omega_{\mathrm d}$ denoting the reference point on the rotation
axis introduced in Subsection~\ref{subsec:disk_model}.

The equations \eqref{eq:constraint_disk} and
\eqref{eq:constraint_ampere} do not contain temporal derivatives. They are
therefore algebraic constraints which have to be incorporated into the
solution concept rather than treated as evolution equations. In particular,
\eqref{eq:constraint_ampere} encodes Ampere's law on the conducting part of
the electromagnet, on the disk, and on the non-conducting region.

\section{Analytical setting and weak solutions}\label{sec:solution_concept}

This section fixes the analytical framework for the subsequent solvability
analysis by collecting the standing hypotheses and introducing a weak solution
concept compatible with the algebraic constraints of the magneto-quasistatic
model.

\subsection{Notation and function spaces}
\label{subsec:prelim}

All spaces considered throughout this paper are real Hilbert spaces. The inner
product and norm on a Hilbert space $X$ are written as
$\langle\cdot,\cdot\rangle_X$ and $\|\cdot\|_X$, respectively. For
$a,b\in\R^n$, we write $a^\top b$ for the Euclidean inner product and
$\|a\|_2:=(a^\top a)^{1/2}$ for the Euclidean norm. For vector-valued fields, this notation is used pointwise. For Hilbert spaces
$X$ and $Y$, we write
$\mathcal L(X,Y)$ for the space of bounded linear operators from $X$ to $Y$,
and simply $\mathcal L(X)$ if $X=Y$. The identity operator on $X$ is denoted by
$\id_X$, or simply by $\id$ if the underlying space is clear. The notation $X\hookrightarrow Y$ means that
$X$ is continuously embedded in $Y$, i.e., $X\subset Y$ and the identity map
$X\to Y$ is continuous.
For a possibly unbounded linear operator
$A:\dom(A)\subset X\to Y$, we denote its domain by $\dom(A)$. If $A$ is
densely defined, then its adjoint is denoted by
$A^*:\dom(A^*)\subset Y\to X$.

We follow the standard notation for Lebesgue and Sobolev spaces used
in~\cite{AdamsFournier2003}. For function spaces with values in a Hilbert
space $X$, we append ``$;X$'' to the underlying domain. Thus
$L^p(\Omega;X)$ denotes the space of $p$-integrable $X$-valued functions on
$\Omega$. Throughout the paper, integrals of $X$-valued functions are
understood in the Bochner sense; see~\cite{DiestelUhl1977}. If $T\in(0,\infty]$ and $X$ is a Banach space, then
$C_0^1((0,T);X)$ denotes the space of continuously differentiable
$X$-valued functions with compact support in the open interval $(0,T)$.

For an open set $D\subset\R^3$ and
$\bm F\in L^2(D;\R^3)$, the symbol $\curl\bm F$ denotes the weak curl whenever
it exists. We use the standard space $H(\curl,D)$ of square-integrable vector
fields with square-integrable weak curl; see~\cite{Monk2003}.

If $D\subset\R^3$ is a bounded Lipschitz domain, we denote by
$\bm n_D$ the outward unit normal vector on $\partial D$. If $D=\Omega$,
we write $\bm n_o:=\bm n_\Omega$. For smooth fields on $D$, we use the
tangential trace notation
\[
\gamma_\tau\bm F:=\bm F\times\bm n_D,
\qquad
\gamma_T\bm F:=\bm n_D\times(\bm F\times\bm n_D).
\]
These traces extend in the usual weak sense to $H(\curl,D)$ fields with values
in the corresponding Maxwell trace spaces; see
\cite[Sec.~4, Thm.~4.1]{Buffa2002}.

We also use the space of square-integrable tangential boundary fields
\[
L^2_t(\partial D)
:=
\{\bm g\in L^2(\partial D;\R^3):
\bm g^\top\bm n_D=0\ \text{a.e. on }\partial D\}.
\]

\subsection{Standing hypotheses}
\label{subsec:standing_assumptions}

Throughout the paper, unless explicitly stated otherwise, we impose the
following hypotheses.
\begin{enumerate}[(a)]
\item\label{hyp:geometry} We assume that $\Omega\subset\R^3$ is a bounded
Lipschitz domain. The subdomains $\Omega_{\mathrm d}$ and
$\Omega_{\mathrm{em}}$ are open, disjoint Lipschitz subdomains compactly
contained in $\Omega$. We set
$\Omega_{\mathrm c}:=\Omega_{\mathrm d}\cup\Omega_{\mathrm{em}}$,
$\Omega_0:=\Omega\setminus\overline{\Omega_{\mathrm c}}$.
Moreover, $\Omega_0$ is assumed to be connected with Lipschitz boundary.
\item\label{hyp:reluctivity} The magnetic reluctivity $\nu:\Omega\to\R^{3\times3}$ is 
pointwise symmetric and positive definite, and satisfies
$\nu,\nu^{-1}\in L^\infty(\Omega;\R^{3\times3})$.
\item\label{hyp:conductivity} The conductivities
$\sigma_{\mathrm d}:\Omega_{\mathrm d}\to\R^{3\times3}$ and
$\sigma_{\mathrm{em}}:\Omega_{\mathrm{em}}\to\R^{3\times3}$ are pointwise symmetric and
positive definite, with
$\sigma_{\mathrm d},\sigma_{\mathrm d}^{-1}
\in L^\infty(\Omega_{\mathrm d};\R^{3\times3})$, 
$\sigma_{\mathrm{em}},\sigma_{\mathrm{em}}^{-1}
\in L^\infty(\Omega_{\mathrm{em}};\R^{3\times3})$.
\item\label{hyp:boundary_absorption} The boundary absorption coefficient
satisfies
$\zeta\in W^{1,\infty}(\partial\Omega)$,
$\zeta\geq0$ a.e. on $\partial\Omega$.
\item\label{hyp:inputs} The winding density satisfies
$\chi\in L^\infty(\Omega;\R^{3\times m})$ and
$\supp\chi\subseteq\Omega_{\mathrm{em}}$.
The injected current and the externally applied torque satisfy
$i\in L^2_{\mathrm{loc}}(\R_{\ge0};\R^m)$,
$M_{\ext}\in L^1_{\mathrm{loc}}(\R_{\ge0};\R)$.
\item\label{hyp:disk} The disk is modeled as a rigid axisymmetric body rotating about the
$x$-axis with moment of inertia $I_{\mathrm d}>0$.
\end{enumerate}
\begin{remark}[Physical interpretation of the standing hypotheses]
\label{rem:physical_interpretation_hypotheses}
The standing hypotheses have the following physical interpretation.
\begin{enumerate}[(a)]
\item
Hypothesis~\eqref{hyp:geometry} imposes no essential geometric restriction
for the setting considered here: the idealized disk and electromagnet
geometries are Lipschitz, and the artificial computational domain $\Omega$
can be chosen Lipschitz.
\item
Hypotheses~\eqref{hyp:reluctivity} and~\eqref{hyp:conductivity} describe
linear, possibly anisotropic material laws. The symmetry and uniform positivity
of $\nu$ correspond to a positive magnetic energy density. The symmetry and
uniform positivity of $\sigma_{\mathrm d}$ and $\sigma_{\mathrm{em}}$
correspond to positive Ohmic dissipation in the disk and in the conducting
part of the electromagnet. Degenerate or singular material laws, such as
vanishing reluctivity, insulating behavior in the conducting subdomains, or
perfect conductivity, are excluded.
\item
Hypothesis~\eqref{hyp:boundary_absorption} describes the absorbing boundary
condition on the artificial outer boundary. The coefficient $\zeta$ is a
nonnegative scalar boundary absorption coefficient with the physical dimension
of a resistance, and its $W^{1,\infty}$-regularity is an analytical
compatibility condition for the weak formulation.
\item
Hypothesis~\eqref{hyp:inputs} specifies the external actuation. The winding
density $\chi$ maps the input current $i$ to the impressed current density
$\bm J_{\mathrm{em}}=\chi i$ and is supported in the modeled electromagnet.
The scalar input $M_{\ext}$ is the externally applied mechanical torque.
\item
Hypothesis~\eqref{hyp:disk} is the rigid-body approximation for the moving
conductor. The angular velocity is related to the angular momentum by
$\omega=I_{\mathrm d}^{-1}L$, and the vector field
$\bm e_x\times\bm r$ introduced in Subsection~\ref{subsec:disk_model} is the
spatial part of the corresponding rigid-body velocity field.
\end{enumerate}
\end{remark}

We next record several conventions and consequences of the standing
hypotheses. Since $\Omega_{\mathrm d}$ is bounded by
Hypothesis~\eqref{hyp:geometry}, the vector field introduced in
Subsection~\ref{subsec:disk_model} satisfies
$\xi\mapsto \bm e_x\times\bm r(\xi)
\in L^\infty(\Omega_{\mathrm d};\R^3)$.
Based on Hypotheses~\eqref{hyp:geometry} and~\eqref{hyp:conductivity}, we
define the conductivity on the conducting region $\Omega_{\mathrm c}$ by
\begin{equation}
\label{eq:sigma_on_omega_c}
\sigma
:=
\charFunction{\Omega_{\mathrm d}}\sigma_{\mathrm d}
+
\charFunction{\Omega_{\mathrm{em}}}\sigma_{\mathrm{em}}
\quad\text{on }\Omega_{\mathrm c},
\end{equation}
where $\sigma_{\mathrm d}$ and $\sigma_{\mathrm{em}}$ are understood as
extended by zero to $\Omega_{\mathrm c}$. Then $\sigma$ is symmetric and
uniformly positive definite on $\Omega_{\mathrm c}$. Moreover,
\[
\sigma,\sigma^{-1}\in L^\infty(\Omega_{\mathrm c};\R^{3\times3}),
\qquad
\sigma^{-1}
=
\charFunction{\Omega_{\mathrm d}}\sigma_{\mathrm d}^{-1}
+
\charFunction{\Omega_{\mathrm{em}}}\sigma_{\mathrm{em}}^{-1}
\quad\text{on }\Omega_{\mathrm c}.
\]
The constant boundary absorption coefficients mentioned in
Remark~\ref{rem:physical_interpretation_hypotheses} are covered by
Hypothesis~\eqref{hyp:boundary_absorption}. In particular, the lossless case
$\zeta\equiv0$ is included. Then the Silver--Müller-type boundary
condition reduces to the homogeneous tangential boundary condition
$\gamma_\tau\bm E=0$ on $\partial\Omega$.

Finally, by \eqref{eq:sigma_on_omega_c}, $\sigma$ is symmetric and uniformly
positive definite on $\Omega_{\mathrm c}$. Hence the pointwise symmetric
square roots $\sigma^{1/2}$ and $\sigma^{-1/2}$ are well defined and belong
to $L^\infty(\Omega_{\mathrm c};\R^{3\times3})$.

\subsection{Magnetic constraints, finite-dissipation spaces, and weak solutions}
\label{subsec:weak_solution}

We first introduce the preliminary class of magnetic fields with finite
instantaneous dissipation,
\[
\mathcal D_{\mathrm{fd}}
:=
\setdef{
\bm B\in L^2(\Omega;\R^3)}{
\nu\bm B\in H(\curl,\Omega),\,\gamma_T(\nu\bm B)\in L^2_t(\partial\Omega)
\text{ and }\curl(\nu\bm B)|_{\Omega_0}=0
\text{ distributionally}
}.
\]
We define the magnetic energy space and its inner product by
\begin{equation*}
X_B
:=
\overline{\mathcal D_{\mathrm{fd}}}^{\,L^2(\Omega;\R^3)},
\qquad
\langle \bm B,\bm C\rangle_{X_B}
:=
\int_\Omega
(\nu\bm B)^\top\bm C\,\mathrm d\xi.
\end{equation*}
By Hypothesis~\eqref{hyp:reluctivity}, this inner product induces a norm
equivalent to the $L^2(\Omega;\R^3)$-norm. Thus $X_B$ is a Hilbert space.
Moreover, since the set of all
$\bm B\in L^2(\Omega;\R^3)$ satisfying
$\curl(\nu\bm B)=0$ in $\mathcal D'(\Omega_0;\R^3)$ is closed in
$L^2(\Omega;\R^3)$, every element of $X_B$ satisfies this magnetic constraint
in $\Omega_0$.

The boundary dissipation is described by the weighted tangential Hilbert space
$L^2_t(\partial\Omega,\zeta\,\mathrm dS)$,
that is, the space of tangential fields which are square-integrable with
respect to the measure $\zeta\,\mathrm dS$, where fields which agree
$\zeta\,\mathrm dS$-almost everywhere are identified. Its inner product is
\[
\langle \bm h,\bm k\rangle_{L^2_t(\partial\Omega,\zeta\,\mathrm dS)}
:=
\int_{\partial\Omega}
\zeta\,\bm h^\top\bm k\,\mathrm dS.
\]
The lossless case $\zeta\equiv0$ is included; in this case
$L^2_t(\partial\Omega,\zeta\,\mathrm dS)\cong\{0\}$.

We define
\[
\mathcal V_B^0
:=
\setdef{
\left(
\bm B,
\sigma^{-1/2}
\left.\curl(\nu\bm B)\right|_{\Omega_{\mathrm c}},
\gamma_T(\nu\bm B)
\right)}{\bm B\in\mathcal D_{\mathrm{fd}}
}
\subset
X_B\times L^2(\Omega_{\mathrm c};\R^3)
\times L^2_t(\partial\Omega,\zeta\,\mathrm dS)
\]
and set
\begin{align*}
V_B
&:=
\overline{\mathcal V_B^0}^{\,X_B\times L^2(\Omega_{\mathrm c};\R^3)
\times L^2_t(\partial\Omega,\zeta\,\mathrm dS)},\\
\|\mathfrak B\|_{V_B}
&:=
\big(
\|\bm B\|_{X_B}^2
+
\|\bm q\|_{L^2(\Omega_{\mathrm c};\R^3)}^2
+
\|\bm h\|_{L^2_t(\partial\Omega,\zeta\,\mathrm dS)}^2
\big)^{1/2},
\quad
\mathfrak B=(\bm B,\bm q,\bm h).
\end{align*}
Thus $V_B$ is a Hilbert space. The first component $\bm B$ represents the
magnetic flux density, the second component $\bm q$ represents the
conductivity-weighted conductive current, and the third component $\bm h$
represents the tangential magnetic boundary field. On the dense subspace
$\mathcal V_B^0$ these components are given by
$\bm q
=
\sigma^{-1/2}\left.\curl(\nu\bm B)\right|_{\Omega_{\mathrm c}}$,
$\bm h
=
\gamma_T(\nu\bm B)$.
In particular, $X_B$ represents magnetic states with finite magnetic energy,
whereas $V_B$ represents magnetic states together with finite bulk and
boundary dissipation variables.

\begin{lemma}[Uniqueness of finite-dissipation representatives]
\label{lem:finite_dissipation_representatives_unique}
Let $(0,\bm q,\bm h)\in V_B$. Then $\bm q=0$ and
$\zeta\bm h=0$ almost everywhere on $\partial\Omega$. In particular,
$\bm h=0$ in $L^2_t(\partial\Omega,\zeta\,\mathrm dS)$.
\end{lemma}

\begin{proof}
By the definition of $V_B$, there exists a sequence
$(\bm B_n)_{n\in\natNum}\subset\mathcal D_{\mathrm{fd}}$ such that
\[\bm B_n\to0\;\text{ in }X_B,\qquad
\sigma^{-1/2}\left.\curl(\nu\bm B_n)\right|_{\Omega_{\mathrm c}}
\to \bm q
\;\text{ in }L^2(\Omega_{\mathrm c};\R^3),\qquad
\gamma_T(\nu\bm B_n)\to\bm h
\;\text{ in }L^2_t(\partial\Omega,\zeta\,\mathrm dS).
\]
Set $\bm H_n:=\nu\bm B_n$. Since the $X_B$-norm is equivalent to the
$L^2(\Omega;\R^3)$-norm on $X_B$ and
$\nu\in L^\infty(\Omega;\R^{3\times3})$, we have
$\bm H_n\to0$ in $L^2(\Omega;\R^3)$.
Moreover,
$\left.\curl\bm H_n\right|_{\Omega_{\mathrm c}}
\to \sigma^{1/2}\bm q$ in $L^2(\Omega_{\mathrm c};\R^3)$.
Since $\Omega_{\mathrm c}=\Omega_{\mathrm d}\cup\Omega_{\mathrm{em}}$ is a
disjoint union, this convergence holds on both
$\Omega_{\mathrm d}$ and $\Omega_{\mathrm{em}}$. The curl operator is closed on
each of these open sets. Hence
$\sigma^{1/2}\bm q\vert_{\Omega_{\mathrm d}}=0$ and
$\sigma^{1/2}\bm q\vert_{\Omega_{\mathrm{em}}}=0$.
Thus $\sigma^{1/2}\bm q=0$ in $L^2(\Omega_{\mathrm c};\R^3)$.
Since $\sigma^{1/2}$ is uniformly positive definite on $\Omega_{\mathrm c}$,
we obtain $\bm q=0$.

It remains to identify the boundary limit. Since
$\bm B_n\in\mathcal D_{\mathrm{fd}}$, we have
$\curl\bm H_n\vert_{\Omega_0}=0$ distributionally. Together with the
convergence already shown, this gives
$\curl\bm H_n\to0$ in $L^2(\Omega;\R^3)$, hence
$\bm H_n\to0$ in $H(\curl;\Omega)$. By the continuity of the weak tangential
trace, see \cite[Theorem~3.29]{Monk2003},
$\gamma_T\bm H_n\to0$ in $H^{-1/2}(\partial\Omega;\R^3)$.

On the other hand, $\gamma_T\bm H_n\to\bm h$ in
$L^2_t(\partial\Omega,\zeta\,\mathrm dS)$. Let $\bm\psi$ be a smooth boundary
test field. Since $\zeta\in W^{1,\infty}(\partial\Omega)$, multiplication by
$\zeta$ preserves $H^{1/2}(\partial\Omega;\R^3)$, and hence
$\zeta\bm\psi\in H^{1/2}(\partial\Omega;\R^3)$. Therefore
\[
\int_{\partial\Omega}\zeta\bm h^\top\bm\psi\,\mathrm dS
=
\lim_{n\to\infty}
\int_{\partial\Omega}\zeta(\gamma_T\bm H_n)^\top\bm\psi\,\mathrm dS
=
0,
\]
where the last equality follows from
$\gamma_T\bm H_n\to0$ in $H^{-1/2}(\partial\Omega;\R^3)$. Thus
$\zeta\bm h=0$ distributionally on $\partial\Omega$. Since
$\bm h\in L^2_t(\partial\Omega,\zeta\,\mathrm dS)$ and
$\zeta\in L^\infty(\partial\Omega)$, we have
$\zeta\bm h\in L^1(\partial\Omega;\R^3)$. Hence
$\zeta\bm h=0$ almost everywhere. This is equivalent to
$\bm h=0$ in $L^2_t(\partial\Omega,\zeta\,\mathrm dS)$.
\end{proof}
By Lemma~\ref{lem:finite_dissipation_representatives_unique}, the projection
$(\bm B,\bm q,\bm h)\mapsto \bm B$
is injective on $V_B$. Whenever the Hilbert triple
$V_B\hookrightarrow X_B\hookrightarrow V_B'$
is used, we identify $V_B$ with this first-component image. Under this
identification, the second and third components of an element of $V_B$ are
still denoted by $\bm q$ and $\bm h$.

For the weak formulation of Faraday's law we use the test space
\[
\mathcal W
:=
\setdef{
\bm\varphi\in H^1(\Omega;\R^3)}{
(\curl\bm\varphi)\vert_{\Omega_0}=0
}.
\]
The space $\mathcal W$ is a closed subspace of $H^1(\Omega;\R^3)$ and is
equipped with the induced $H^1$-norm. This choice reflects that the weak
solution concept below contains the electric field only on the conducting
region $\Omega_{\mathrm c}$. Hence Faraday's law is tested only against
variations whose curl vanishes in the non-conducting subdomain $\Omega_0$.

We shall also use the canonical embedding of $X_B$ into $V_B'$. Namely, for
$\bm Z\in X_B$ and $\mathfrak C=(\bm C,\bm p,\bm k)\in V_B$, we set
$\langle \bm Z,\mathfrak C\rangle_{V_B',V_B}
:=
\langle \bm Z,\bm C\rangle_{X_B}$.
\begin{definition}[Weak solution of the eddy current brake]
\label{def:weak_solution}
Assume Hypotheses~\ref{hyp:geometry}--\ref{hyp:disk}. Let
$T\in(0,\infty)$ and let $(\bm B_0,L_0)\in X_B\times\R$ be given. A
quadruple $(\mathfrak B,L,\bm E,\bm J_{\mathrm d})$,
$\mathfrak B=(\bm B,\bm q,\bm h)$,
is called a weak solution of the eddy current brake on $[0,T]$ if
\begin{align*}
\mathfrak B&\in L^2([0,T];V_B),
&\bm B&\in H^1([0,T];V_B')\cap C([0,T];X_B),
&L&\in W^{1,1}([0,T]),\\
\bm E&\in L^2([0,T];L^2(\Omega_{\mathrm c};\R^3)),
&\bm J_{\mathrm d}&\in L^2([0,T];L^2(\Omega_{\mathrm d};\R^3)),
\end{align*}
with $\bm B(0)=\bm B_0$ and $L(0)=L_0$, and if the following conditions hold.

First, Faraday's law together with the Silver--Müller boundary condition is
satisfied in the finite-dissipation form sense: for almost every
$t\in(0,T)$ and every
$\mathfrak C=(\bm C,\bm p,\bm k)\in V_B$,
\[
\frac{\mathrm d}{\mathrm dt}
\langle \bm B(t),\bm C\rangle_{X_B}
+
\int_{\Omega_{\mathrm c}}
\bigl(\sigma^{1/2}\bm E(t)\bigr)^\top\bm p
\,\mathrm d\xi
+
\int_{\partial\Omega}
\zeta
\bm h(t)^\top\bm k
\,\mathrm dS
=
0.
\]
Second, the algebraic constraints in the conducting subdomains are satisfied for almost every $t\in(0,T)$:
\begin{align*}
\sigma_{\mathrm{em}}^{1/2}
\left.\bm q(t)\right|_{\Omega_{\mathrm{em}}}
&=
\sigma_{\mathrm{em}}
\left.\bm E(t)\right|_{\Omega_{\mathrm{em}}}
+
\chi i(t)
&&\text{in }L^2(\Omega_{\mathrm{em}};\R^3),\\
\bm J_{\mathrm d}(t)
=
\sigma_{\mathrm d}^{1/2}
\left.\bm q(t)\right|_{\Omega_{\mathrm d}}
&=
\sigma_{\mathrm d}
\left(
\left.\bm E(t)\right|_{\Omega_{\mathrm d}}
+
I_{\mathrm d}^{-1}L(t)(\bm e_x\times\bm r)
\times
\left.\bm B(t)\right|_{\Omega_{\mathrm d}}
\right)
&&\text{in }L^2(\Omega_{\mathrm d};\R^3).
\end{align*}
Third, the angular momentum balance is satisfied in the sense that, for almost every $t\in(0,T)$,
\[
\dot L(t)
=
\bm e_x^\top
\int_{\Omega_{\mathrm d}}
\bm r(\xi)\times
\bigl(\bm J_{\mathrm d}(\xi,t)\times\bm B(\xi,t)\bigr)
\,\mathrm d\xi
+
M_{\ext}(t).
\]
Moreover, a quadruple
$(\mathfrak B,L,\bm E,\bm J_{\mathrm d})$
is called a weak solution on $\R_{\ge0}$ if its restriction to $[0,T]$ is a
weak solution in the sense of Definition~\ref{def:weak_solution} for every
finite $T>0$.
\end{definition}
\begin{proposition}[Energy balance for weak solutions]
\label{prop:energy_balance_weak_solution}
Let $(\mathfrak B,L,\bm E,\bm J_{\mathrm d})$ be a weak solution on $[0,T]$
in the sense of Definition~\ref{def:weak_solution}. Then, for all
$t\in[0,T]$,
\[
\begin{split}
\mathcal E(t)
&+
\int_0^t
\left(
\|\bm q(\tau)\|_{L^2(\Omega_{\mathrm c};\R^3)}^2
+
\|\bm h(\tau)\|_{L^2_t(\partial\Omega,\zeta\,\mathrm dS)}^2
\right)
\,\mathrm d\tau
\\
&=
\mathcal E(0)
+
\int_0^t
\int_{\Omega_{\mathrm{em}}}
\bigl(\sigma^{-1/2}\chi i(\tau)\bigr)^\top
\left.\bm q(\tau)\right|_{\Omega_{\mathrm{em}}}
\,\mathrm d\xi\,\mathrm d\tau
+
\int_0^t
I_{\mathrm d}^{-1}L(\tau)M_{\ext}(\tau)
\,\mathrm d\tau,
\end{split}
\]
where
$\mathcal E(t)
:=
\tfrac12\|\bm B(t)\|_{X_B}^2
+
\tfrac{1}{2I_{\mathrm d}}L(t)^2$.
\end{proposition}

\begin{proof}
Since $V_B$ is separable, the variational identity can be realized on a
single full-measure subset of $(0,T)$ for all test functions in $V_B$: first
one proves this on a countable dense subset of $V_B$, and then extends it by
continuity in the test function. Hence, for almost every $t\in(0,T)$, we may
choose the test function $\mathfrak C=\mathfrak B(t)$ in the
finite-dissipation form of Faraday's law and obtain
\[
\frac12\frac{\mathrm d}{\mathrm dt}
\|\bm B(t)\|_{X_B}^2
+
\int_{\Omega_{\mathrm c}}
\bigl(\sigma^{1/2}\bm E(t)\bigr)^\top\bm q(t)
\,\mathrm d\xi
+
\|\bm h(t)\|_{L^2_t(\partial\Omega,\zeta\,\mathrm dS)}^2
=
0.
\]
Here we used the Hilbert-triple chain rule in the above first-component
identification of $V_B$; see
\cite[Ch.~XVIII, Sec.~1, Thm.~2]{DautrayLionsVol5}.

We rewrite the conductive term by using the algebraic constraints. On
$\Omega_{\mathrm{em}}$ we have
\[
\left.\bm q(t)\right|_{\Omega_{\mathrm{em}}}
=
\sigma_{\mathrm{em}}^{1/2}
\left.\bm E(t)\right|_{\Omega_{\mathrm{em}}}
+
\sigma_{\mathrm{em}}^{-1/2}\chi i(t),
\]
and hence
\[
\bigl(\sigma_{\mathrm{em}}^{1/2}\bm E(t)\bigr)^\top
\left.\bm q(t)\right|_{\Omega_{\mathrm{em}}}
=
\left\|
\left.\bm q(t)\right|_{\Omega_{\mathrm{em}}}
\right\|_2^2
-
\bigl(\sigma_{\mathrm{em}}^{-1/2}\chi i(t)\bigr)^\top
\left.\bm q(t)\right|_{\Omega_{\mathrm{em}}}.
\]
On $\Omega_{\mathrm d}$ the disk constraint gives
\[
\left.\bm q(t)\right|_{\Omega_{\mathrm d}}
=
\sigma_{\mathrm d}^{1/2}
\left(
\left.\bm E(t)\right|_{\Omega_{\mathrm d}}
+
\bm v(t)\times\left.\bm B(t)\right|_{\Omega_{\mathrm d}}
\right),
\]
where $\bm v(t)=I_{\mathrm d}^{-1}L(t)(\bm e_x\times\bm r)$. Therefore
\[
\bigl(\sigma_{\mathrm d}^{1/2}\bm E(t)\bigr)^\top
\left.\bm q(t)\right|_{\Omega_{\mathrm d}}
=
\left\|
\left.\bm q(t)\right|_{\Omega_{\mathrm d}}
\right\|_2^2
-
\bigl(
\bm v(t)\times\left.\bm B(t)\right|_{\Omega_{\mathrm d}}
\bigr)^\top
\bm J_{\mathrm d}(t).
\]
Consequently,
\begin{multline*}
\frac12\frac{\mathrm d}{\mathrm dt}
\|\bm B(t)\|_{X_B}^2
+
\|\bm q(t)\|_{L^2(\Omega_{\mathrm c};\R^3)}^2
+
\|\bm h(t)\|_{L^2_t(\partial\Omega,\zeta\,\mathrm dS)}^2\\
=
\int_{\Omega_{\mathrm{em}}}
\bigl(\sigma_{\mathrm{em}}^{-1/2}\chi i(t)\bigr)^\top
\left.\bm q(t)\right|_{\Omega_{\mathrm{em}}}
\,\mathrm d\xi
+
\int_{\Omega_{\mathrm d}}
\bigl(
\bm v(t)\times\left.\bm B(t)\right|_{\Omega_{\mathrm d}}
\bigr)^\top
\bm J_{\mathrm d}(t)
\,\mathrm d\xi.
\end{multline*}
The angular momentum balance gives
\[
\frac{\mathrm d}{\mathrm dt}
\frac{1}{2I_{\mathrm d}}L(t)^2
=
I_{\mathrm d}^{-1}L(t)
\bm e_x^\top
\int_{\Omega_{\mathrm d}}
\bm r(\xi)\times
\bigl(\bm J_{\mathrm d}(\xi,t)\times\bm B(\xi,t)\bigr)
\,\mathrm d\xi
+
I_{\mathrm d}^{-1}L(t)M_{\ext}(t).
\]
Using $\bm v(t)=I_{\mathrm d}^{-1}L(t)(\bm e_x\times\bm r)$ and the scalar
triple product identity, we obtain
\[
I_{\mathrm d}^{-1}L(t)
\bm e_x^\top
\int_{\Omega_{\mathrm d}}
\bm r\times
\bigl(\bm J_{\mathrm d}\times\bm B\bigr)
\,\mathrm d\xi
=
-
\int_{\Omega_{\mathrm d}}
\bigl(
\bm v(t)\times\left.\bm B(t)\right|_{\Omega_{\mathrm d}}
\bigr)^\top
\bm J_{\mathrm d}(t)
\,\mathrm d\xi.
\]
Thus the Lorentz coupling terms cancel in the sum of the magnetic and
mechanical energy identities. Integrating the resulting differential identity
from $0$ to $t$ gives the asserted balance.
\end{proof}
\begin{remark}[Well-definedness of the weak formulation]
All terms in Definition~\ref{def:weak_solution} are well defined. Indeed,
$\mathfrak B\in L^2([0,T];V_B)$ implies
$\bm B\in L^2([0,T];X_B)$, $\bm q\in L^2([0,T];L^2(\Omega_{\mathrm c};\R^3))$, and
$\bm h\in L^2([0,T];L^2_t(\partial\Omega,\zeta\,\mathrm dS))$. Moreover, for every $\mathfrak C=(\bm C,\bm p,\bm k)\in V_B$, the function
$t\mapsto\langle\bm B(t),\bm C\rangle_{X_B}$ is absolutely continuous and
its derivative is
$\langle\dot{\bm B}(t),\mathfrak C\rangle_{V_B',V_B}$ for almost every
$t\in(0,T)$. Hence the first term in the finite-dissipation form of
Faraday's law is well defined.

If $\mathfrak C=(\bm C,\bm p,\bm k)\in V_B$, then
$\bm p\in L^2(\Omega_{\mathrm c};\R^3)$ and $\bm k\in L^2_t(\partial\Omega,\zeta\,\mathrm dS)$. Since
$\bm E\in L^2([0,T];L^2(\Omega_{\mathrm c};\R^3))$ and
$\sigma^{1/2}\in L^\infty(\Omega_{\mathrm c};\R^{3\times3})$, the conductive
term in the form identity is well defined. The boundary term is precisely the
inner product in $L^2_t(\partial\Omega,\zeta\,\mathrm dS)$.

The algebraic disk constraint is meaningful because
\[
I_{\mathrm d}^{-1}L(\bm e_x\times\bm r)\times\bm B
\in L^2([0,T];L^2(\Omega_{\mathrm d};\R^3))
\]
whenever $L\in C([0,T])$ and $\bm B\in L^2([0,T];X_B)$. Finally, the torque term
belongs to $L^1([0,T])$ because $\bm r\in L^\infty(\Omega_{\mathrm d};\R^3)$,
$\bm B\in C([0,T];X_B)$, and
$\bm J_{\mathrm d}\in L^2([0,T];L^2(\Omega_{\mathrm d};\R^3))$.
\end{remark}

\begin{remark}[Relation to the space-time Faraday formulation]
Let $(\mathfrak B,L,\bm E,\bm J_{\mathrm d})$ be a weak solution in the sense
of Definition~\ref{def:weak_solution}. Then, for every $\bm\varphi\in C_0^1((0,T);\mathcal W)$, the space-time Faraday identity
\begin{multline*}
-\int_0^T
\int_\Omega
\bm B(\xi,t)^\top\tfrac{\mathrm{d}}{\mathrm{d}t}\bm\varphi(\xi,t)
\,\mathrm d\xi\,\mathrm dt
+
\int_0^T
\int_{\Omega_{\mathrm c}}
\bm E(\xi,t)^\top\curl\bm\varphi(\xi,t)
\,\mathrm d\xi\,\mathrm dt
\\
+
\int_0^T
\int_{\partial\Omega}
\zeta(\xi)
\bm h(\xi,t)^\top\gamma_T\bm\varphi(\xi,t)
\,\mathrm dS(\xi)\,\mathrm dt
=
0.
\end{multline*}
Indeed, for such a
test function one has
\[
\left(
\nu^{-1}\bm\varphi,
\sigma^{-1/2}\left.\curl\bm\varphi\right|_{\Omega_{\mathrm c}},
\gamma_T\bm\varphi
\right)
\in V_B,
\]
and the assertion follows by inserting this test function into the
finite-dissipation form of Faraday's law and integrating in time.
\end{remark}

\begin{remark}[Propagation of the Maxwell constraints]
\label{rem:maxwell_constraints}
The weak evolution preserves the magnetic Gauss constraint: the distributional
divergence of $\bm B$ is constant in time, that is,
$\divg\bm B(t)=\divg\bm B_0$ for all $t\in[0,T]$. In particular, if
$\divg\bm B_0=0$, then $\divg\bm B(t)=0$ for every $t\in[0,T]$.

Indeed, in the space-time Faraday identity one may choose test functions of
the form $\bm\varphi(\xi,t)=\eta(t)\nabla\psi(\xi)$, where
$\eta\in C_0^1((0,T))$ and $\psi\in C_c^\infty(\Omega)$. Since
$\curl\nabla\psi=0$ and $\nabla\psi$ vanishes in a neighborhood of
$\partial\Omega$, both the electric-field term and the boundary term vanish,
and hence
\[
-\int_0^T \eta'(t)
\int_\Omega \bm B(\xi,t)^\top\nabla\psi(\xi)\,\mathrm d\xi\,\mathrm dt
=0.
\]
Thus, for every $\psi\in C_c^\infty(\Omega)$, the map
$t\mapsto\int_\Omega\bm B(\xi,t)^\top\nabla\psi(\xi)\,\mathrm d\xi$ is
constant. Since $\bm B\in C([0,T];X_B)$, the $X_B$-norm is equivalent to the
$L^2(\Omega;\mathbb R^3)$-norm, and $\bm B(0)=\bm B_0$, this yields
$\divg\bm B(t)=\divg\bm B_0$ for every $t\in[0,T]$.

Likewise, taking the distributional divergence of Ampère's law and using
$\divg\curl=0$ shows that the total current density
\[
\bm J_{\mathrm{tot}}(t)
:=
\charFunction{\Omega_{\mathrm{em}}}
\bigl(\sigma_{\mathrm{em}}\bm E(t)+\chi i(t)\bigr)
+
\charFunction{\Omega_{\mathrm d}}\bm J_{\mathrm d}(t)
\]
satisfies $\divg\bm J_{\mathrm{tot}}(t)=0$ for almost every
$t\in(0,T)$. Thus no additional $H(\divg)$-regularity is required in
Definition~\ref{def:weak_solution}.\end{remark}

\begin{remark}[Physical interpretation of the weak solution concept]
The weak solution is formulated for the variables which occur in the
partial differential-algebraic model: the magnetic flux density $\bm B$, the
angular momentum $L$, the electric field $\bm E$ on the conducting region,
and the disk current density $\bm J_{\mathrm d}$. The finite-dissipation
representative $\mathfrak B=(\bm B,\bm q,\bm h)$ additionally records the
conductivity-weighted conductive current $\bm q$ and the tangential magnetic
boundary variable $\bm h$.

The constraint on $\Omega_{\mathrm{em}}$ is Ampere's law in the conducting
part of the electromagnet. The constraint on $\Omega_{\mathrm d}$ combines
Ampere's law in the disk with Ohm's law including the motional Lorentz term.
The magnetic constraint in the non-conducting subdomain is encoded in the
definition of $X_B$. The restriction in the test space $\mathcal W$ reflects
that the electric field is included as an unknown only on the conducting
region $\Omega_{\mathrm c}$.
\end{remark}

\begin{remark}[Consistency with classical solutions]
Every sufficiently regular solution of the pointwise model
\eqref{eq:model} is a weak solution in the sense of
Definition~\ref{def:weak_solution}. Indeed, assume that
$(\bm B,L,\bm E,\bm J_{\mathrm d})$ is sufficiently regular, satisfies
\eqref{eq:model} together with the Silver--Müller boundary condition, and that
$\bm B(t)\in\mathcal D_{\mathrm{fd}}$ for almost every $t$. Then
\[
\mathfrak B(t)
:=
\left(
\bm B(t),
\sigma^{-1/2}\left.\curl(\nu\bm B(t))\right|_{\Omega_{\mathrm c}},
\gamma_T(\nu\bm B(t))
\right)
\in V_B
\]
for almost every $t$. Moreover, the second component of $\mathfrak B(t)$ satisfies
\begin{align*}
\sigma_{\mathrm{em}}^{1/2}
\left.
\left(
\sigma^{-1/2}\curl(\nu\bm B(t))
\right)
\right|_{\Omega_{\mathrm{em}}}
&=
\left.\curl(\nu\bm B(t))\right|_{\Omega_{\mathrm{em}}}
=
\sigma_{\mathrm{em}}\bm E|_{\Omega_{\mathrm{em}}}(t)+\chi i(t),\\
\sigma_{\mathrm d}^{1/2}
\left.
\left(
\sigma^{-1/2}\curl(\nu\bm B(t))
\right)
\right|_{\Omega_{\mathrm d}}
&=
\left.\curl(\nu\bm B(t))\right|_{\Omega_{\mathrm d}}
=
\bm J_{\mathrm d}(t)=
\sigma_{\mathrm d}
\left(
\bm E|_{\Omega_{\mathrm d}}(t)
+
I_{\mathrm d}^{-1}L(t)(\bm e_x\times\bm r)
\times
\bm B|_{\Omega_{\mathrm d}}(t)
\right).
\end{align*}
Testing Faraday's law with
$\bm\varphi\in C_0^1((0,T);\mathcal W)$, using integration by parts in space
and time and inserting the Silver--Müller boundary condition, gives the
space-time Faraday identity stated above. Equivalently, in pointwise-in-time
form, this is the finite-dissipation form of Faraday's law in
Definition~\ref{def:weak_solution}. The angular momentum balance in
\eqref{eq:model} gives the angular momentum balance in
Definition~\ref{def:weak_solution}.
\end{remark}

\subsection{Port-Hamiltonian modeling}
\label{subsec:pH_modeling}

We reformulate the coupled PDAE \eqref{eq:model} in the port-Hamiltonian framework to obtain an energy-based representation of the system. 
The central idea of this approach is that the power, defined as the temporal rate of change of the total energy, can be expressed in terms of dual variables called ports.
Recasting the governing equations in these variables exposes the system's energy-routing structure, making transparent how power is exchanged between the electromagnet and the mechanical brake, where it is supplied from external sources, and where it is dissipated. We emphasize that the analysis in this subsection is strictly on a formal level. The rigorous well-posedness analysis presented in the subsequent sections remains based on the weak formulation introduced above.

We first compute the time derivative of the total energy
$\mathcal E$:
\[
\dot{\mathcal{E}}(t)
=
\langle \bm B(t),\tfrac{\mathrm{d}}{\mathrm{d}t} \bm B(t)\rangle_{X_B}
+
\left(I_{\mathrm d}{}^{-1}L(t)\right) \dot L(t)\,.
\]
With respect to the $L^2(\Omega;\mathbb R^3)$-pairing and the Euclidean
pairing, respectively, the corresponding variational derivatives of
$\mathcal E$ are $\nu\bm B(t)$ and $I_{\mathrm d}^{-1}L(t)$.
Using Faraday's law \eqref{eq:dynB}, integration by parts, and the boundary
condition \eqref{eq:model_boundary}, we obtain
\[
\dot{\mathcal{E}}(t) = - \langle \gamma_T(\nu \bm B(t)), \gamma_T(\nu\bm B(t)) \rangle_{L^2_t(\partial\Omega,\zeta\,\mathrm dS)}- \langle \curl (\nu \bm B(t)), \bm E(t) \rangle_{L^2(\Omega_{\mathrm c};\mathbb{R}^3)}+
\left(I_{\mathrm d}{}^{-1}L(t)\right) \dot L(t)\,,
\]
where boundary condition \eqref{eq:model_boundary} has been inserted into the trace term. Next, we split the integral domain $\Omega_{\mathrm{c}}=\Omega_{\mathrm{em}} \cup \Omega_{\mathrm{d}}$
according to Hypothesis~\ref{hyp:geometry}. Substituting Ampère's law \eqref{eq:constraint_ampere}, the relation
\eqref{eq:model_abbreviations}, and the equation of motion
\eqref{eq:dynL}, the total power is given by
\begin{align*}
\dot{\mathcal{E}}(t) = &- \langle \gamma_T(\nu \bm B(t)), \gamma_T(\nu\bm B(t)) \rangle_{L^2_t(\partial\Omega,\zeta\,\mathrm dS)}- \langle \sigma_{\mathrm{em}} \bm E(t), \bm E(t) \rangle_{L^2(\Omega_{\mathrm{em}}; \mathbb{R}^3)}  \\
&- \langle \chi i(t), \bm E(t) \rangle_{L^2(\Omega_{\mathrm{em}}; \mathbb{R}^3)}- \langle \bm J_{\mathrm d}(t), \bm E(t) \rangle_{L^2(\Omega_{\mathrm{d}}; \mathbb{R}^3)}
+ \left(I_{\mathrm d}{}^{-1}L(t)\right) \mathcal T_{{\bm B}(t)}\bm J_{\mathrm d}(t) + \left(I_{\mathrm d}{}^{-1}L(t)\right)M_{\ext}(t)
\end{align*}
where, for a given magnetic field $\bm B$, we define the torque operator,
\[
\mathcal T_{\bm B(t)}: L^2(\Omega_{\mathrm{d}};\mathbb{R}^3) \to \mathbb{R} \,, \quad \bm F \mapsto \bm e_x^\top
\int_{\Omega_{\mathrm d}}
\bm r(\xi)\times
\bigl(\bm F(\xi)\times\bm B(t, \xi)\bigr)
\,\mathrm d\xi \,,
\]
so that
$\dot L(t)
=
\mathcal T_{\bm B(t)}\bm J_{\mathrm d}(t)
+
M_{\ext}(t)$.
The adjoint of the torque operator is determined by
\[
\left\langle
\mathcal T_{\bm B(t)}^*\omega,\bm J
\right\rangle_{L^2(\Omega_{\mathrm d};\mathbb R^3)}
=
\omega\,\mathcal T_{\bm B(t)}\bm J,
\qquad
\bm J\in L^2(\Omega_{\mathrm d};\mathbb R^3).
\]
Using the scalar triple-product identity, we obtain
\[
\mathcal T_{\bm B}^*\omega
=
-\omega(\bm e_x\times\bm r)\times\bm B.
\]
In particular, for
$\omega(t)=I_{\mathrm d}^{-1}L(t)$ and
$\bm v(t)=\omega(t)(\bm e_x\times\bm r)$,
\[
\mathcal T_{{\bm B}(t)}^*
\left(I_{\mathrm d}^{-1}L(t)\right)
=
-\bm v(t)\times\bm B(t).
\]
Consequently,
\[
\left(I_{\mathrm d}^{-1}L(t)\right)
\mathcal T_{{\bm B}(t)}\bm J_{\mathrm d}(t)
+
\left\langle
\bm v(t)\times\bm B(t),\bm J_{\mathrm d}(t)
\right\rangle_{L^2(\Omega_{\mathrm d};\mathbb R^3)}
=
0.
\]
Using
\[
\left.\bm E(t)\right|_{\Omega_{\mathrm d}}
=
\sigma_{\mathrm d}^{-1}\bm J_{\mathrm d}(t)
-
\bm v(t)\times\bm B(t),
\]
the electromagnetic and mechanical Lorentz power terms cancel, and we obtain
the formal power balance
\begin{equation}
\label{eq:formalPowerBalance}
\begin{split}
\dot{\mathcal E}(t)
&=
-
\left\|
\gamma_T\bigl(\nu\bm B(t)\bigr)
\right\|_{L^2_t(\partial\Omega,\zeta\,\mathrm dS)}^2
-
\left\langle
\sigma_{\mathrm{em}}\bm E(t),\bm E(t)
\right\rangle_{L^2(\Omega_{\mathrm{em}};\R^3)}
\\
&\quad
-
\left\langle
\bm J_{\mathrm d}(t),
\sigma_{\mathrm d}^{-1}\bm J_{\mathrm d}(t)
\right\rangle_{L^2(\Omega_{\mathrm d};\R^3)}
-
i(t)^\top
\int_{\Omega_{\mathrm{em}}}
\chi^\top\bm E(t)\,\mathrm d\xi
+
I_{\mathrm d}^{-1}L(t)M_{\ext}(t).
\end{split}
\end{equation}
The first three terms describe boundary and bulk dissipation, whereas the
last two terms are the electrical and mechanical external port powers.

To systematically separate power-conserving interconnection from losses, we reformulate the coupled PDAE \eqref{eq:model} in terms of a Dirac structure
and a resistive structure~\cite{vanderSchaft2014}. For this, first note that we can formally rewrite the coupled PDAE \eqref{eq:model} system as
\begin{equation}
\label{eq:SEDAE}
    \begin{pmatrix}
        - \tfrac{\mathrm{d}}{\mathrm{d}t} \bm B(t) \\
        - \dot L (t) \\
        f_{\mathrm{R},1}(t) \\
        f_{\mathrm{R},2}(t) \\
        f_\partial(t)  \\
        f_{\ext, 1}(t) \\
        f_{\ext, 2}(t) \\
        \nu \bm B(t) \\
        \omega (t) \\
        e_{\mathrm{R},1}(t) \\
        e_{\mathrm{R},2}(t) \\
        e_\partial(t)  \\
        e_{\ext, 1}(t) \\
        e_{\ext, 2}(t) 
    \end{pmatrix} = 
    \begin{pmatrix}
        0 & 0 &  \curl & 0 & 0 & 0  \\
        0 & 0 &  0 & - \mathcal T_{{\bm B}(t)}  & - 1 & 0 \\
        - \curl & 0 & 0 & \charFunction{\Omega_{\mathrm d}} & 0 & 1  \\
        0 & \mathcal{T}_{{\bm B}(t)}{}^{*}  & - \cdot \mid_{\Omega_{\mathrm{d}}} & 0 & 0 & 0 \\
        0 & 0 & \gamma_\tau & 0 & 0 & 0  \\
        0 & 0 & 0 & 0 & 1 & 0   \\
        0 & 0 & 0 & 0 & 0 & 1 \\
        1 & 0 & 0 & 0 & 0 & 0 \\
        0 & 1 & 0 & 0 & 0 & 0 \\
        0 & 0 & 1 & 0 & 0 & 0 \\
        0 & 0 & 0 & 1 & 0 & 0 \\
        \gamma_T & 0 & 0 & 0 & 0 & 0 \\
        0 & 1 & 0 & 0 & 0 & 0 \\
        0 & 0 & -1 & 0 & 0 & 0
    \end{pmatrix}
    \begin{pmatrix}
        \nu \bm B(t) \\
        \omega (t) \\
        \bm E(t) \\
        \bm J_{\mathrm{d}}(t) \\
        M_{\ext }(t) \\
        \chi i(t)
    \end{pmatrix}, \;\;
    \begin{pmatrix}
        f_{\mathrm{R},1}(t) \\
        f_{\mathrm{R},2}(t) \\
        f_\partial (t)
    \end{pmatrix} = -
    \begin{pmatrix}
\charFunction{\Omega_{\mathrm{em}}}\sigma_{\mathrm{em}} & 0 & 0\\
         0 & \sigma_{\mathrm{d}}{}^{-1} & 0 \\
         0 & 0 &  \zeta
    \end{pmatrix}
    \begin{pmatrix}
        e_{\mathrm{R},1}(t) \\
        e_{\mathrm{R},2}(t) \\
        e_\partial (t)
    \end{pmatrix},
\end{equation}
splitting off the resistive from the power-conserving part of the dynamics, and identifying explicitly the external port variables appearing in the power balance \eqref{eq:formalPowerBalance}. Here $\charFunction{\Omega_{\mathrm d}}\bm J_{\mathrm d}$ denotes the extension of the disk current $\bm J_{\mathrm{d}}$ by zero to $\Omega$.

This representation motivates us to define, analogous to the image  framework of \cite[Theorem~4.8]{BehrndtKurulaVanDerSchaftZwart2010} but extended to include external and resistive ports, the subspace
\begin{equation*}
\mathcal{D}_{\bm B} \coloneqq \im 
    \begin{pmatrix}
    \mathcal{J}_{\bm B} & - \beta^* \\
        \left[ \begin{smallmatrix} 0 & 0 & \gamma_\tau & 0 \end{smallmatrix} \right] & 0  \\
        0 & I_{2 \times 2} \\
        I_{4 \times 4} & 0 \\
        \left[ \begin{smallmatrix} \gamma_T & 0 & 0 & 0\ \end{smallmatrix} \right] & 0 \\
        \beta & 0
    \end{pmatrix}
\end{equation*}
given by the image of the matrix operator appearing in \eqref{eq:SEDAE} and where we used 
\begin{equation*}
\mathcal{J}_{\bm B} = \begin{bmatrix}
0 & 0 &  \curl & 0   \\
        0 & 0 &  0 & - \mathcal T_{{\bm B}}  \\
        - \curl & 0 & 0 & \charFunction{\Omega_{\mathrm d}} \\
        0 & \mathcal{T}_{{\bm B}}{}^{*}  & - \cdot \mid_{\Omega_{\mathrm{d}}} & 0
\end{bmatrix} \qquad \mathrm{and} \qquad \beta = \begin{bmatrix}
0 & 1 & 0 & 0 \\
0 & 0 & -1 & 0 
\end{bmatrix}.
\end{equation*}
Here the identity blocks act on the corresponding product spaces, and
$\beta^*$ denotes the adjoint with respect to the associated Hilbert-space
pairings.

On sufficiently regular fields, the image construction of
\cite[Theorem~4.8]{BehrndtKurulaVanDerSchaftZwart2010}, together with
integration by parts in the $\curl$ terms, formally yields a
power-conserving Dirac relation with respect to the corresponding
$L^2$-pairings for the distributed bulk ports, the standard Euclidean
pairing for the finite-dimensional ports, and the $L^2$-pairing on
$\partial\Omega$ for the boundary port $(f_\partial,e_\partial)$.
Note that $\mathcal D_{\bm B}$ depends on the magnetic field $\bm B$ through
the torque operator $\mathcal T_{\bm B}$ and its adjoint. It therefore formally defines a Dirac structure modulated by the magnetic
state, which is one way in which nonlinearities arise in port-Hamiltonian
systems.

Furthermore, the second relation in \eqref{eq:SEDAE} defines the resistive structure
\begin{equation*}
\mathcal{R} \coloneqq \Big\lbrace (f,e) \, \Big|\, f = -
    \left(\begin{smallmatrix}
\charFunction{\Omega_{\mathrm{em}}}\sigma_{\mathrm{em}} & 0 & 0\\
         0 & \sigma_{\mathrm{d}}{}^{-1} & 0 \\
         0 & 0 &  \zeta
    \end{smallmatrix}\right)
    e \Big\rbrace.
\end{equation*}

Combining the modulated Dirac structure $\mathcal D_{\bm B}$ with the
resistive relation $\mathcal R$ and the Hamiltonian $\mathcal E$ yields a
formal port-Hamiltonian representation of the coupled PDAE
\eqref{eq:model}. The original equations can therefore be expressed in the standard port-Hamiltonian form
\begin{equation*}
\begin{gathered}
\left(- \tfrac{\mathrm{d}}{\mathrm{d}t} \bm B(t),
        - \dot L (t),
        f_{\mathrm{R},1}(t),
        f_{\mathrm{R},2}(t),
        f_\partial(t)  ,
%        f_{\ext, 1}(t) ,
%        f_{\ext, 2}(t) ,
		M_{\ext }(t) ,
        \chi i(t) ,
        \frac{\delta \mathcal{E}}{\delta \bm B} ,
        \frac{\delta \mathcal{E}}{\delta L}    ,
        e_{\mathrm{R},1}(t) ,
        e_{\mathrm{R},2}(t) ,
        e_\partial(t)  ,
        e_{\ext, 1}(t) ,
        e_{\ext, 2}(t) 
    \right)^T \in \mathcal{D}_{\bm B(t)} \,, \\
  \left(f_{\mathrm{R},1}(t),
        f_{\mathrm{R},2}(t),
        f_\partial(t)  ,
        e_{\mathrm{R},1}(t) ,
        e_{\mathrm{R},2}(t) ,
        e_\partial(t)  
    \right)^T \in \mathcal{R}
\end{gathered}
\end{equation*}
thereby making the exchange, dissipation, and supply of energy explicit.

\section{The magneto--quasistatic subsystem}\label{sec:mqs}

This section constructs the linear magneto--quasistatic realization used in
the solvability analysis below. The construction is based on the
finite-dissipation space $V_B$ introduced in
Subsection~\ref{subsec:weak_solution}. It yields a nonnegative closed form,
the associated self-adjoint operator, and the corresponding inhomogeneous
evolution equation in the extrapolation space $V_B'$.

\subsection{The main operator and its properties}
\label{sec:contraction}

We now construct the operator governing the homogeneous magnetic-field
dynamics. Under the first-component identification introduced after
Lemma~\ref{lem:finite_dissipation_representatives_unique}, define
\begin{equation}
\label{eq:finite_dissipation_operator}
K:V_B\subset X_B
\to
L^2(\Omega_{\mathrm c};\R^3)
\times L^2_t(\partial\Omega,\zeta\,\mathrm dS),
\qquad
K\bm B:=(\bm q,\bm h)
\end{equation}
whenever $(\bm B,\bm q,\bm h)\in V_B$.
This is well defined and linear by
Lemma~\ref{lem:finite_dissipation_representatives_unique}. Moreover, $K$ is
closed, since its graph is precisely the closed subspace $V_B$ under the
identification above. It is densely defined because
$\mathcal D_{\mathrm{fd}}$ is dense in $X_B$. Finally, if
$K\bm B=(\bm q,\bm h)$, then the graph norm of $K$ is the $V_B$-norm,
\[
\|\bm B\|_{V_B}^2
=
\|\bm B\|_{X_B}^2
+
\|\bm q\|_{L^2(\Omega_{\mathrm c};\R^3)}^2
+
\|\bm h\|_{L^2_t(\partial\Omega,\zeta\,\mathrm dS)}^2.
\]
Define the symmetric nonnegative bilinear form
\begin{equation*}
\mathfrak a(\bm B,\bm C)
:=
\int_{\Omega_{\mathrm c}}
\bm q_{\bm B}^{\top}\bm q_{\bm C}\,\mathrm d\xi
+
\int_{\partial\Omega}
\zeta\,\bm h_{\bm B}^{\top}\bm h_{\bm C}\,\mathrm dS,
\qquad
\dom(\mathfrak a):=V_B,
\end{equation*}
where $K\bm B=(\bm q_{\bm B},\bm h_{\bm B})$ and
$K\bm C=(\bm q_{\bm C},\bm h_{\bm C})$. Since $K$ is closed, the form
$\mathfrak a$ is closed on $X_B$.
\begin{theorem}[Self-adjoint MQS realization]
\label{thm:selfadjoint_realization}
Let $G$ be the operator associated with the closed form $\mathfrak a$, that is,
\[
\dom(G)
:=
\left\{
\bm B\in\dom(K):
\exists\,\bm F\in X_B\ \text{such that }
\mathfrak a(\bm B,\bm C)
=
\langle \bm F,\bm C\rangle_{X_B}
\quad\forall\,\bm C\in\dom(K)
\right\},
\]
and $G\bm B:=\bm F$. Then $G=K^*K$.
In particular, $G$ is self-adjoint and nonnegative on $X_B$. Hence
$A:=-G=-K^*K$
is self-adjoint, dissipative, and maximally dissipative on $X_B$.
Consequently, $A$ generates a strongly continuous contraction semigroup
$(T(t))_{t\geq0}$ on $X_B$.
\end{theorem}
\begin{proof}
Since $K$ is closed and densely defined, the form $\mathfrak a$ is closed,
symmetric, and nonnegative. By the first representation theorem for closed
nonnegative symmetric forms, the associated operator $G$ is self-adjoint and
nonnegative; see \cite[Ch.~VI, Sec.~2, Thm.~2.1]{Kato1995}. Moreover, for
forms of the type
$\mathfrak a(\bm B,\bm C)=\langle K\bm B,K\bm C\rangle$ the associated
operator is $K^*K$; see \cite[Ch.~VI, Sec.~2.4, Example~2.13]{Kato1995}.
Hence $G=K^*K$.

It follows that $A=-G$ is self-adjoint and dissipative. Moreover, for every
$\lambda>0$, the operator $\lambda\id+G$ is onto $X_B$. Equivalently,
$\lambda\id-A$ is onto $X_B$. Hence $A$ is maximally dissipative. The
Lumer--Phillips theorem yields that $A$ generates a strongly continuous
contraction semigroup on $X_B$; see
\cite[Thm.~II.3.15]{EngelNagel2000}.
\end{proof}
We now pass to the form realization in the Hilbert triple
$V_B\hookrightarrow X_B\hookrightarrow V_B'$,
where $X_B$ is used as pivot space. The form $\mathfrak a$ induces the bounded operator
\begin{equation*}
\mathcal A:V_B\to V_B',
\qquad
\langle \mathcal A\bm B,\bm C\rangle_{V_B',V_B}
=
\mathfrak a(\bm B,\bm C).
\end{equation*}
The self-adjoint operator $G=K^*K$ is the part of $\mathcal A$ in $X_B$,
and $A=-G$ is the corresponding semigroup generator.

We use the associated Hilbert scale
\begin{equation}
\label{eq:hilbert_scale}
X_{1/2}:=V_B,
\qquad
X_{-1/2}:=V_B'.
\end{equation}
Since $A=-G$ is self-adjoint and nonpositive, the semigroup $T(t)$
leaves $V_B=\dom((\id +G)^{1/2})$ invariant and is strongly continuous on
$V_B$. By duality with respect to the pivot space $X_B$, it therefore
induces a strongly continuous semigroup
$(T_{-1/2}(t))_{t\geq0}$ on $X_{-1/2}=V_B'$. This is the standard
interpolation/extrapolation setting for analytic semigroups; cf.
\cite[Part~II, Ch.~3, Sec.~2.1, Thm.~2.2]{BensoussanDaPratoDelfourMitter2007}.
Its generator is the extrapolated operator $A_{-1/2}$, and its restriction to
$X_{1/2}=V_B$ is given by
$A_{-1/2}\bm B=-\mathcal A\bm B$, $\bm B\in V_B$.
\begin{theorem}[Inhomogeneous MQS equation]
\label{thm:inhomogeneous_mqs_equation}
Let $T>0$, $\bm B_0\in X_B$, and
$f\in L^2([0,T];V_B')$. Then there exists a unique
\[
\bm B\in L^2([0,T];V_B)\cap H^1([0,T];V_B')\cap C([0,T];X_B)
\]
satisfying
\[
\dot{\bm B}(t)+\mathcal A\bm B(t)=f(t)
\quad\text{in }V_B'\text{ for a.e. }t\in(0,T),
\qquad
\bm B(0)=\bm B_0.
\]
Moreover, $\bm B$ is given by the variation-of-constants formula
\[
\bm B(t)
=
T(t)\bm B_0
+
\int_0^t
T_{-1/2}(t-\tau)f(\tau)\,\mathrm d\tau,
\qquad t\in[0,T],
\]
where the integral and the equality are understood in $X_{-1/2}=V_B'$.
By the regularity asserted above, the right-hand side belongs to $X_B$ for
every $t\in[0,T]$. The solution satisfies
the estimate
\[
\|\bm B\|_{C([0,T];X_B)}
+
\|\bm B\|_{L^2([0,T];V_B)}
+
\|\dot{\bm B}\|_{L^2([0,T];V_B')}
\leq
C_T
\left(
\|\bm B_0\|_{X_B}
+
\|f\|_{L^2([0,T];V_B')}
\right).
\]
In particular, the solution depends continuously on the initial value and on
the inhomogeneity: if $\bm B_1,\bm B_2$ correspond to
$(\bm B_{0,1},f_1)$ and $(\bm B_{0,2},f_2)$, respectively, then
\[
\|\bm B_1-\bm B_2\|_{C([0,T];X_B)}
+
\|\bm B_1-\bm B_2\|_{L^2([0,T];V_B)}
\leq
C_T
\left(
\|\bm B_{0,1}-\bm B_{0,2}\|_{X_B}
+
\|f_1-f_2\|_{L^2([0,T];V_B')}
\right).
\]
If $\bm B_0=0$, then, for every $T_0>0$, there is a constant $C_{T_0}>0$
such that for all $0<\delta\leq T_0$ and all
$f\in L^\infty([0,\delta];V_B')$,
\[
\|\bm B\|_{C([0,\delta];X_B)}
\leq
C_{T_0}\sqrt{\delta}\,
\|f\|_{L^\infty([0,\delta];V_B')}.
\]
\end{theorem}

\begin{proof}
The existence and uniqueness statement follows from the standard theorem for
variational evolution equations in a Hilbert triple, applied to
$V_B\hookrightarrow X_B\hookrightarrow V_B'$
and to the autonomous form associated with $\mathcal A$; see
\cite[Part~II, Ch.~2, Sec.~1.3, Thm.~1.1]{BensoussanDaPratoDelfourMitter2007}.
This gives
\[
\bm B\in L^2([0,T];V_B)\cap H^1([0,T];V_B')
\cap C([0,T];X_B).
\]
The variation-of-constants formula is the corresponding mild representation
in the extrapolation space $V_B'$; cf. the interpolation/extrapolation
framework for analytic semigroups in
\cite[Part~II, Ch.~3, Thm.~2.2]{BensoussanDaPratoDelfourMitter2007}. Testing
$\dot{\bm B}(t)+\mathcal A\bm B(t)=f(t)$
with $\bm B(t)$, using the Hilbert-triple chain rule
\cite[Ch.~XVIII, Sec.~1, Thm.~2]{DautrayLionsVol5}, and applying Young's inequality yields
\[
\tfrac12\tfrac{\mathrm d}{\mathrm dt}\|\bm B(t)\|_{X_B}^2
+
\mathfrak a(\bm B(t),\bm B(t))
\leq
\langle f(t),\bm B(t)\rangle_{V_B',V_B}
\leq
\tfrac12\|\bm B(t)\|_{V_B}^2
+
\tfrac12\|f(t)\|_{V_B'}^2.
\]
Since
\[
\|\bm B(t)\|_{V_B}^2
=
\|\bm B(t)\|_{X_B}^2+\mathfrak a(\bm B(t),\bm B(t)),
\]
we obtain
\[
\tfrac12\frac{\mathrm d}{\mathrm dt}\|\bm B(t)\|_{X_B}^2
+
\tfrac12\mathfrak a(\bm B(t),\bm B(t))
\leq
\tfrac12\|\bm B(t)\|_{X_B}^2
+
\tfrac12\|f(t)\|_{V_B'}^2.
\]
Gronwall's lemma gives the asserted bounds for
$\|\bm B\|_{C([0,T];X_B)}$ and $\|\bm B\|_{L^2([0,T];V_B)}$.
The bound for $\dot{\bm B}$ follows from
$\dot{\bm B}=f-\mathcal A\bm B$
and from the boundedness of $\mathcal A:V_B\to V_B'$.

The estimate for the difference
of two solutions follows by applying the same argument to the difference
equation. The last estimate follows from the preceding estimate and
\[
\|f\|_{L^2([0,\delta];V_B')}
\leq
\sqrt{\delta}\,
\|f\|_{L^\infty([0,\delta];V_B')}.\qedhere\]
\end{proof}

The following elementary observation connects the abstract form equation with
the test functions used in Definition~\ref{def:weak_solution}. If
$\bm\varphi\in\mathcal W$, then $\nu^{-1}\bm\varphi\in\mathcal D_{\mathrm{fd}}$ and
\[
K(\nu^{-1}\bm\varphi)
=
\left(
\sigma^{-1/2}
\left.\curl\bm\varphi\right|_{\Omega_{\mathrm c}},
\gamma_T\bm\varphi
\right).
\]
Indeed, $\nu(\nu^{-1}\bm\varphi)=\bm\varphi$. Since
$\bm\varphi\in H^1(\Omega;\R^3)$, the standard trace theorem gives
$\gamma\bm\varphi\in H^{1/2}(\partial\Omega;\R^3)\subset
L^2(\partial\Omega;\R^3)$; see \cite[Thm.~3.37]{McLean2000}. Since
$\bm n_o\in L^\infty(\partial\Omega;\R^3)$, this implies
$\gamma_T\bm\varphi\in L^2_t(\partial\Omega)$. Together with
$\curl\bm\varphi=0$ in $\Omega_0$, this proves
$\nu^{-1}\bm\varphi\in\mathcal D_{\mathrm{fd}}$.
Moreover, the mapping
$\mathcal W\to V_B$,
$\bm\varphi\mapsto \nu^{-1}\bm\varphi$,
is bounded.

\begin{lemma}[Weak test identity for the inhomogeneous MQS equation]\label{lem:mqs_test_function_identity}
Let $T>0$, $\bm B_0\in X_B$, and
$f\in L^2([0,T];V_B')$. Let $\bm B$ be the solution from
Theorem~\ref{thm:inhomogeneous_mqs_equation}. Then, for every
$\bm\Phi\in C^1([0,T];V_B)$ with $\bm\Phi(T)=0$,
\[
\int_0^T
\langle \bm B(t),\dot{\bm\Phi}(t)\rangle_{X_B}
\,\mathrm dt
+
\langle \bm B_0,\bm\Phi(0)\rangle_{X_B}
=
\int_0^T
\mathfrak a(\bm B(t),\bm\Phi(t))
\,\mathrm dt
-
\int_0^T
\langle f(t),\bm\Phi(t)\rangle_{V_B',V_B}
\,\mathrm dt.
\]
In particular, for every
$\bm\varphi\in C_0^1((0,T);\mathcal W)$,
\[
\int_0^T
\mathfrak a\bigl(\bm B(t),\nu^{-1}\bm\varphi(t)\bigr)
\,\mathrm dt
=
\int_0^T\int_\Omega
\bm B(\xi,t)^\top\tfrac{\mathrm{d}}{\mathrm{d}t}\bm\varphi(\xi,t)
\,\mathrm d\xi\,\mathrm dt
+\int_0^T
\left\langle
f(t),\nu^{-1}\bm\varphi(t)
\right\rangle_{V_B',V_B}
\,\mathrm dt
.
\]
\end{lemma}

\begin{proof}
For almost every $t\in(0,T)$ and every $\bm C\in V_B$, the evolution
equation is equivalent to
\[
\langle \dot{\bm B}(t),\bm C\rangle_{V_B',V_B}
+
\mathfrak a(\bm B(t),\bm C)
=
\langle f(t),\bm C\rangle_{V_B',V_B}.
\]
Testing this identity with $\bm C=\bm\Phi(t)$, integrating in time, and
using the Hilbert-triple integration-by-parts formula
\cite[Ch.~XVIII, Sec.~1, Thm.~2]{DautrayLionsVol5}
gives the first identity.
For the second identity, choose
$\bm\Phi(t)=\nu^{-1}\bm\varphi(t)$. Since
$\bm\varphi\in C_0^1((0,T);\mathcal W)$, we have
$\bm\Phi\in C^1([0,T];V_B)$ and $\bm\Phi(0)=\bm\Phi(T)=0$. Moreover,
by the symmetry of $\nu$,
\[
\langle \bm B(t),\nu^{-1}\tfrac{\mathrm{d}}{\mathrm{d}t}\bm\varphi(t)\rangle_{X_B}
=
\int_\Omega
\bm B(\xi,t)^\top\tfrac{\mathrm{d}}{\mathrm{d}t}\bm\varphi(\xi,t)
\,\mathrm d\xi.
\]
This proves the asserted identity.
\end{proof}

For later use, we also record the impressed-current operator. For
$\iota\in\R^m$, define $\mathcal B_{\mathrm{em}}\iota\in V_B'$ by
\begin{equation}
\label{eq:em_input_operator}
\left\langle \mathcal B_{\mathrm{em}}\iota,\bm C\right\rangle_{V_B',V_B}
:=
\int_{\Omega_{\mathrm{em}}}
(\sigma^{-1/2}\chi \iota)^\top (K\bm C)_1
\,\mathrm d\xi,
\qquad
\bm C\in V_B,
\end{equation}
where $(K\bm C)_1$ denotes the first component of $K\bm C$ in
$L^2(\Omega_{\mathrm c};\R^3)$. Then
$\mathcal B_{\mathrm{em}}\in\mathcal L(\R^m,V_B')$. Indeed,
\[
\left|
\left\langle \mathcal B_{\mathrm{em}}\iota,\bm C\right\rangle_{V_B',V_B}
\right|
\leq
\|\sigma^{-1/2}\chi \iota\|_{L^2(\Omega_{\mathrm{em}})}
\|(K\bm C)_1\|_{L^2(\Omega_{\mathrm c})}
\leq
c\|\iota\|_2\|\bm C\|_{V_B}.
\]
Moreover, for $\bm\varphi\in\mathcal W$,
\[
\left\langle
\mathcal B_{\mathrm{em}}\iota,\nu^{-1}\bm\varphi
\right\rangle_{V_B',V_B}
=
\int_{\Omega_{\mathrm{em}}}
(\sigma^{-1}\chi \iota)^\top
\left.\curl\bm\varphi\right|_{\Omega_{\mathrm{em}}}
\,\mathrm d\xi.
\]

\section{Solvability of the coupled eddy current brake}
\label{sec:solvability_analysis}

The aim of this section is to prove existence and uniqueness of weak solutions
of the coupled eddy current brake. The proof proceeds in three steps. First,
we rewrite the coupled weak formulation as an abstract integral system for the
pair $(\bm B,L)$. Second, we solve this integral system locally by a
fixed-point argument. Third, we use the energy balance to exclude finite-time
blow-up and hence obtain global solutions.

Throughout this section, the standing hypotheses from
Subsection~\ref{subsec:standing_assumptions} are in force. We use the
magneto-quasistatic realization from Section~\ref{sec:mqs}. In particular,
$A=-K^*K$ generates the contraction semigroup $(T(t))_{t\geq0}$ on $X_B$,
and the inhomogeneous magnetic equation is understood in the form space
$V_B'$.

We isolate the part of the finite-dissipation representative which enters the
disk equation. Set
\begin{equation*}
Y_{\mathrm d}:=L^2(\Omega_{\mathrm d};\R^3).
\end{equation*}
The disk current map is
\begin{equation}
\label{eq:disk_current_observation}
\mathcal C_{\mathrm d}:V_B\to Y_{\mathrm d},
\qquad
\mathcal C_{\mathrm d}\bm B
:=
\left.\sigma^{1/2}(K\bm B)_1\right|_{\Omega_{\mathrm d}}.
\end{equation}
Thus $\mathcal C_{\mathrm d}\bm B$ is the disk current density associated
with the magnetic state $\bm B$. In particular, for
$\bm B\in\mathcal D_{\mathrm{fd}}$,
we have $\mathcal C_{\mathrm d}\bm B
=
\curl(\nu\bm B)\vert_{\Omega_{\mathrm d}}$.
The map $\mathcal C_{\mathrm d}$ is bounded, since $K$ is bounded from
$V_B$ into
$L^2(\Omega_{\mathrm c};\R^3)
\times L^2_t(\partial\Omega,\zeta\,\mathrm dS)$ with respect to the
$V_B$-norm and
$\sigma^{1/2}\in L^\infty(\Omega_{\mathrm c};\R^{3\times3})$.

The corresponding input map from disk current densities into the magnetic
equation is defined by duality,
\begin{equation}
\label{eq:disk_input_operator}
\left\langle \mathcal B_{\mathrm d}\bm u,\bm C
\right\rangle_{V_B',V_B}
:=
\int_{\Omega_{\mathrm d}}
\bm u^\top
\mathcal C_{\mathrm d}\bm C
\,\mathrm d\xi,
\qquad
\bm C\in V_B.
\end{equation}
This defines a bounded operator
$\mathcal B_{\mathrm d}\in\mathcal L(Y_{\mathrm d},V_B')$.

For the reconstruction of the electric field in the electromagnet we also use
\begin{equation*}
\mathcal C_{\mathrm{em}}:V_B\to L^2(\Omega_{\mathrm{em}};\R^3),
\qquad
\mathcal C_{\mathrm{em}}\bm B
:=
\sigma^{1/2}(K\bm B)_1\big\vert_{\Omega_{\mathrm{em}}}.
\end{equation*}
For $L\in\R$, define the disk velocity field
\begin{equation}
\label{eq:disk_velocity_field}
\bm v_L(\xi)
:=
I_{\mathrm d}^{-1}L(\bm e_x\times\bm r(\xi)),
\qquad
\xi\in\Omega_{\mathrm d}.
\end{equation}
Since $\bm e_x\times\bm r\in L^\infty(\Omega_{\mathrm d};\R^3)$, there is a
constant $c_v>0$ such that
\[
\|\bm v_L\times\bm B\|_{Y_{\mathrm d}}
\leq
c_v |L|\,\|\bm B\|_{X_B}
\quad
\text{for all }\bm B\in X_B,\ L\in\R.
\]
For $\bm B\in X_B$ and $\bm J\in Y_{\mathrm d}$, set
\begin{equation*}
\mathcal T(\bm B,\bm J)
:=
\bm e_x^\top
\int_{\Omega_{\mathrm d}}
\bm r(\xi)\times(\bm J(\xi)\times\bm B(\xi))
\,\mathrm d\xi.
\end{equation*}
Then $\mathcal T$ is a continuous bilinear map. More precisely, there exists
a constant $c_{\mathcal T}>0$ such that
\[
|\mathcal T(\bm B,\bm J)|
\leq
c_{\mathcal T}\|\bm B\|_{X_B}\|\bm J\|_{Y_{\mathrm d}}
\qquad
\text{for all }(\bm B,\bm J)\in X_B\times Y_{\mathrm d}.
\]
For given functions $\bm B$ and $L$, we write
\begin{equation*}
\bm u_{\bm B,L}(t)
:=
\bm v_{L(t)}\times\left.\bm B(t)\right|_{\Omega_{\mathrm d}}.
\end{equation*}
On a time interval $[0,\delta]$, the integrated variation-of-constants system is
\begin{subequations}
\label{eq:coupled_voc_fixed_point}
\begin{align}
\bm B(t)
&=
T(t)\bm B_0
+
\int_0^t
T_{-1/2}(t-\tau)
\left[
\mathcal B_{\mathrm{em}}i(\tau)
+
\mathcal B_{\mathrm d}\bm u_{\bm B,L}(\tau)
\right]
\,\mathrm d\tau,
\label{eq:coupled_voc_fixed_point_B}
\\
L(t)
&=
L_0
+
\int_0^t
\mathcal T\bigl(\bm B(\tau),\mathcal C_{\mathrm d}\bm B(\tau)\bigr)
\,\mathrm d\tau
+
\int_0^t
M_{\ext}(\tau)\,\mathrm d\tau,
\qquad t\in[0,\delta].\label{eq:coupled_voc_fixed_point_L}
\end{align}
\end{subequations}
The first identity is understood in $X_{-1/2}$ for every $t\in[0,\delta]$; for
solutions in the regularity class considered below, both sides belong to
$C([0,\delta];X_B)$.

Given a pair $(\bm B,L)$ satisfying
\[
\bm B\in L^2([0,\delta];V_B)\cap C([0,\delta];X_B),
\qquad
L\in C([0,\delta]),
\]
we reconstruct, for almost every $t\in(0,\delta)$, the variables of
Definition~\ref{def:weak_solution} by
\begin{equation}
\label{eq:weak_solution_reconstruction}
\begin{aligned}
\mathfrak B(t)&:=
\bigl(\bm B(t),(K\bm B(t))_1,(K\bm B(t))_2\bigr),\!\!\!\!\!
&
\bm J_{\mathrm d}(t)&:=\mathcal C_{\mathrm d}\bm B(t),
\\
\left.\bm E(t)\right|_{\Omega_{\mathrm{em}}}
&:=
\sigma_{\mathrm{em}}^{-1}
\left(
\mathcal C_{\mathrm{em}}\bm B(t)-\chi i(t)
\right),
&
\left.\bm E(t)\right|_{\Omega_{\mathrm d}}
&:=
\sigma_{\mathrm d}^{-1}\mathcal C_{\mathrm d}\bm B(t)
-
\bm v_{L(t)}\times\left.\bm B(t)\right|_{\Omega_{\mathrm d}}.
\end{aligned}
\end{equation}
\begin{proposition}[Weak solutions and the variation-of-constants system]
\label{prop:weak_solution_voc_equivalence}
Assume Hypotheses~\ref{hyp:geometry}--\ref{hyp:disk}. Let
$0<\delta<\infty$ and let $(\bm B_0,L_0)\in X_B\times\R$.
Let $(\bm B,L)$ satisfy
\[
\bm B\in L^2([0,\delta];V_B)\cap C([0,\delta];X_B),
\qquad
L\in C([0,\delta]),
\]
and let $(\mathfrak B,L,\bm E,\bm J_{\mathrm d})$ be reconstructed from
$(\bm B,L)$ by \eqref{eq:weak_solution_reconstruction}. Then the following assertions are
equivalent.

\begin{enumerate}[(i)]
\item\label{item:weak_solution_voc_weak}
The reconstructed quadruple
$(\mathfrak B,L,\bm E,\bm J_{\mathrm d})$
is a weak solution on $[0,\delta]$ in the sense of
Definition~\ref{def:weak_solution}.

\item\label{item:weak_solution_voc_integral}
The pair $(\bm B,L)$ satisfies the integrated variation-of-constants system
\eqref{eq:coupled_voc_fixed_point}.
\end{enumerate}
\end{proposition}
\begin{proof}
Assume first that \eqref{item:weak_solution_voc_weak} holds. By
Definition~\ref{def:weak_solution},
$\bm B\in H^1([0,\delta];V_B')$,
$L\in W^{1,1}([0,\delta])$.
Inserting the algebraic constraints into the finite-dissipation form of
Faraday's law gives
\[
\dot{\bm B}(t)+\mathcal A\bm B(t)
=
\mathcal B_{\mathrm{em}}i(t)
+
\mathcal B_{\mathrm d}
\bigl(\bm v_{L(t)}\times\left.\bm B(t)\right|_{\Omega_{\mathrm d}}\bigr)
\quad\text{in }V_B'
\]
for almost every $t\in(0,\delta)$, with $\bm B(0)=\bm B_0$. Moreover, the
angular momentum balance gives
\[
\dot L(t)
=
\mathcal T\bigl(\bm B(t),\mathcal C_{\mathrm d}\bm B(t)\bigr)
+
M_{\ext}(t)
\]
for almost every $t\in(0,\delta)$, with $L(0)=L_0$.

Theorem~\ref{thm:inhomogeneous_mqs_equation} applied to the magnetic form
equation and integration of the scalar equation for $L$ yield
\eqref{eq:coupled_voc_fixed_point}. Hence
\eqref{item:weak_solution_voc_integral} holds.

Conversely, assume that \eqref{item:weak_solution_voc_integral} holds. Since
$i\in L^2([0,\delta];\R^m)$, we have
$\mathcal B_{\mathrm{em}}i\in L^2([0,\delta];V_B')$.
Moreover, since $L\in C([0,\delta])$ and
$\bm B\in C([0,\delta];X_B)$,
\[
\left[
t\mapsto
\bm v_{L(t)}\times\left.\bm B(t)\right|_{\Omega_{\mathrm d}}
\right]
\in L^\infty([0,\delta];Y_{\mathrm d}),
\]
and hence $\bm u_{\bm B,L}\in L^\infty([0,\delta];Y_{\mathrm d})$, $\mathcal B_{\mathrm d}\bm u_{\bm B,L}\in L^2([0,\delta];V_B')$.
Thus Theorem~\ref{thm:inhomogeneous_mqs_equation} implies
\[
\bm B\in L^2([0,\delta];V_B)\cap H^1([0,\delta];V_B')
\cap C([0,\delta];X_B),
\]
and the variation-of-constants identity for $\bm B$ is equivalent to the
magnetic form equation. The scalar integral equation gives
$L\in W^{1,1}([0,\delta])$ and the angular momentum balance almost everywhere.

It remains to check the reconstructed variables. By definition,
\[
\mathfrak B(t)=
\bigl(\bm B(t),(K\bm B(t))_1,(K\bm B(t))_2\bigr)
\]
belongs to $L^2([0,\delta];V_B)$. The reconstruction formulas give
\[
\bm J_{\mathrm d}\in L^2([0,\delta];L^2(\Omega_{\mathrm d};\R^3)),
\qquad
\bm E\in L^2([0,\delta];L^2(\Omega_{\mathrm c};\R^3)),
\]
and they give exactly the algebraic constraints in the electromagnet and in
the disk. The magnetic form equation, after inserting these reconstruction
formulas, is precisely the finite-dissipation form of Faraday's law in
Definition~\ref{def:weak_solution}. Therefore the reconstructed quadruple is a weak solution on $[0,\delta]$. Hence \eqref{item:weak_solution_voc_weak} holds.
\end{proof}
For later use, we record the following standard estimates for the input maps
associated with the form realization.
\begin{lemma}[Basic input-output estimates]
\label{lem:basic_input_output_estimates}
Assume Hypotheses~\ref{hyp:geometry}--\ref{hyp:disk}. Let $A=-K^*K$, where
$K$ is defined by \eqref{eq:finite_dissipation_operator}, and let
$(T(t))_{t\geq0}$ be the contraction semigroup generated by $A$. Let
$(T_{-1/2}(t))_{t\geq0}$ be the extrapolated semigroup on $V_B'$ introduced
in \eqref{eq:hilbert_scale}. Let $\mathcal C_{\mathrm d}$,
$\mathcal B_{\mathrm d}$, and $\mathcal B_{\mathrm{em}}$ be defined by
\eqref{eq:disk_current_observation}, \eqref{eq:disk_input_operator}, and
\eqref{eq:em_input_operator}.

For every $\delta>0$ and every $\bm B_0\in X_B$,
\[
T(\cdot)\bm B_0\in L^2([0,\delta];V_B)\cap C([0,\delta];X_B)
\]
and there exists a constant $c>0$, independent of $\delta$ and $\bm B_0$, such that
\[
\|\mathcal C_{\mathrm d}T(\cdot)\bm B_0\|_{L^2([0,\delta];Y_{\mathrm d})}
\leq
c\|\bm B_0\|_{X_B}.
\]
Moreover, let either
$(U,\mathcal B)=(Y_{\mathrm d},\mathcal B_{\mathrm d})$ or
$(U,\mathcal B)=(\R^m,\mathcal B_{\mathrm{em}})$.
For every $\delta>0$ and every $v\in L^2([0,\delta];U)$, the function
\[
t\mapsto
\bm z_v(t)
:=
\int_0^t
T_{-1/2}(t-\tau)\mathcal B v(\tau)\,\mathrm d\tau
\]
belongs to
\[
L^2([0,\delta];V_B)\cap H^1([0,\delta];V_B')\cap C([0,\delta];X_B).
\]
There is a constant $c>0$, independent of $\delta$ and $v$, and uniform for the
two choices of $(U,\mathcal B)$, such that
\[
\|\bm z_v\|_{C([0,\delta];X_B)}
+
\|\mathcal C_{\mathrm d}\bm z_v\|_{L^2([0,\delta];Y_{\mathrm d})}
\leq
c\|v\|_{L^2([0,\delta];U)}.
\]
In particular, if $v\in L^\infty([0,\delta];U)$, then
\[
\|\bm z_v\|_{C([0,\delta];X_B)}
\leq
c\sqrt{\delta}\,
\|v\|_{L^\infty([0,\delta];U)}.
\]
\end{lemma}

\begin{proof}
The homogeneous regularity follows from
Theorem~\ref{thm:inhomogeneous_mqs_equation} with $f=0$. The energy identity
for the homogeneous form equation gives
\[
\|T(t)\bm B_0\|_{X_B}^2
+
2\int_0^t
\mathfrak a(T(\tau)\bm B_0,T(\tau)\bm B_0)
\,\mathrm d\tau
\leq
\|\bm B_0\|_{X_B}^2,
\qquad t\geq0.
\]
Since
\[
\|\mathcal C_{\mathrm d}\bm C\|_{Y_{\mathrm d}}
\leq
c\,\mathfrak a(\bm C,\bm C)^{1/2},
\qquad \bm C\in V_B,
\]
we obtain
\[
\|\mathcal C_{\mathrm d}T(\cdot)\bm B_0\|_{L^2([0,\delta];Y_{\mathrm d})}
\leq
c\|\bm B_0\|_{X_B}.
\]
The membership $T(\cdot)\bm B_0\in L^2([0,\delta];V_B)$ follows because
$\delta<\infty$, $T(\cdot)\bm B_0\in C([0,\delta];X_B)$, and the preceding
estimate controls the dissipation part.

Let now $(U,\mathcal B)$ be one of the two stated choices and let
$v\in L^2([0,\delta];U)$. By
Theorem~\ref{thm:inhomogeneous_mqs_equation}, the function $\bm z_v$ is the
unique solution of
\[
\dot{\bm z}_v(t)+\mathcal A\bm z_v(t)=\mathcal B v(t)
\quad\text{in }V_B',
\qquad
\bm z_v(0)=0.
\]
Testing this equation with $\bm z_v(t)$ gives
\[
\tfrac12\tfrac{\mathrm d}{\mathrm dt}\|\bm z_v(t)\|_{X_B}^2
+
\mathfrak a(\bm z_v(t),\bm z_v(t))
=
\langle \mathcal B v(t),\bm z_v(t)\rangle_{V_B',V_B}.
\]
For $\mathcal B=\mathcal B_{\mathrm d}$, the definition
\eqref{eq:disk_input_operator} and the estimate for $\mathcal C_{\mathrm d}$
give
\[
\left|
\langle \mathcal B_{\mathrm d}v(t),\bm z_v(t)\rangle_{V_B',V_B}
\right|
\leq
c\|v(t)\|_{Y_{\mathrm d}}\,
\mathfrak a(\bm z_v(t),\bm z_v(t))^{1/2}.
\]
For $\mathcal B=\mathcal B_{\mathrm{em}}$, the definition
\eqref{eq:em_input_operator} gives the same estimate with
$\|v(t)\|_{\R^m}$. Hence, in both cases,
\[
\left|
\langle \mathcal B v(t),\bm z_v(t)\rangle_{V_B',V_B}
\right|
\leq
\tfrac12\mathfrak a(\bm z_v(t),\bm z_v(t))
+
c\|v(t)\|_U^2.
\]
After integration in time, this yields
\[
\|\bm z_v\|_{C([0,\delta];X_B)}^2
+
\int_0^\delta
\mathfrak a(\bm z_v(\tau),\bm z_v(\tau))
\,\mathrm d\tau
\leq
c\|v\|_{L^2([0,\delta];U)}^2.
\]
Since $\delta<\infty$, this implies
$\bm z_v\in L^2([0,\delta];V_B)\cap C([0,\delta];X_B)$. The equation gives
$\bm z_v\in H^1([0,\delta];V_B')$. Finally, using again
$\|\mathcal C_{\mathrm d}\bm C\|_{Y_{\mathrm d}}
\leq
c\,\mathfrak a(\bm C,\bm C)^{1/2}$
gives
\[
\|\bm z_v\|_{C([0,\delta];X_B)}
+
\|\mathcal C_{\mathrm d}\bm z_v\|_{L^2([0,\delta];Y_{\mathrm d})}
\leq
c\|v\|_{L^2([0,\delta];U)}.
\]
The $L^\infty$-estimate follows from
\[
\|v\|_{L^2([0,\delta];U)}
\leq
\sqrt{\delta}\,
\|v\|_{L^\infty([0,\delta];U)}.\qedhere
\]
\end{proof}
\begin{proposition}[Local solvability and blow-up alternative]
\label{prop:local_solvability_integral_system}
Assume Hypotheses~\ref{hyp:geometry}--\ref{hyp:disk}. Let
$(\bm B_0,L_0)\in X_B\times\R$. Then there exists a unique maximal solution
\[
(\bm B,L)\in
\bigl(
L^2_{\mathrm{loc}}([0,t_{\max});V_B)
\cap
C([0,t_{\max});X_B)
\bigr)
\times
C([0,t_{\max}))
\]
of the integrated variation-of-constants system
\eqref{eq:coupled_voc_fixed_point}, where
$t_{\max}\in(0,\infty]$.

More precisely, for every $\delta<t_{\max}$, the restriction of
$(\bm B,L)$ to $[0,\delta]$ satisfies
\eqref{eq:coupled_voc_fixed_point}. If $t_{\max}<\infty$, then
\[
\limsup_{t\nearrow t_{\max}}
\bigl(
\|\bm B(t)\|_{X_B}+|L(t)|
\bigr)
=
\infty.
\]
\end{proposition}
\begin{proof}
Fix $T_*>0$. For $0<\delta\leq T_*$ and
$(\bm b,\ell)\in
C([0,\delta];X_B)\times C([0,\delta];\R)$, define
\[
\bm u_{\bm b,\ell}(t)
:=
\bm v_{\ell(t)}\times
\left.\bm b(t)\right|_{\Omega_{\mathrm d}},
\qquad t\in[0,\delta].
\]
Let
\[
\bm z_i^{T_*}(t)
:=
\int_0^t
T_{-1/2}(t-\tau)\mathcal B_{\mathrm{em}}i(\tau)\,\mathrm d\tau,
\qquad t\in[0,T_*],
\]
and, for $0<\delta\leq T_*$, let
$\bm z_i
:=
\bm z_i^{T_*}\vert_{[0,\delta]}$.
Furthermore, set
\[
\bm z_{\bm b,\ell}(t)
:=
\int_0^t
T_{-1/2}(t-\tau)\mathcal B_{\mathrm d}
\bm u_{\bm b,\ell}(\tau)
\,\mathrm d\tau.
\]
By Lemma~\ref{lem:basic_input_output_estimates},
$\bm z_i,\bm z_{\bm b,\ell}\in C([0,\delta];X_B)$.
Equip
$C([0,\delta];X_B)\times C([0,\delta];\R)$
with the norm
\[
\|(\bm b,\ell)\|_\delta
:=
\max
\left\{
\|\bm b\|_{C([0,\delta];X_B)},
\|\ell\|_{C([0,\delta])}
\right\}.
\]
Choose $R\geq1$ such that
\[
\|\bm B_0\|_{X_B}
+
|L_0|
+
\|\bm z_i^{T_*}\|_{C([0,T_*];X_B)}
+
\|M_{\ext}\|_{L^1([0,T_*])}
\leq R.
\]
For $0<\delta\leq T_*$, define the closed ball
\[
\mathcal S_{\delta,R}
:=
\left\{
(\bm b,\ell)\in
C([0,\delta];X_B)\times C([0,\delta];\R):
\|(\bm b,\ell)\|_\delta\leq2R
\right\}.
\]
For $(\bm b,\ell)\in\mathcal S_{\delta,R}$, define
\begin{align*}
\bm B(t)
&:=
T(t)\bm B_0+\bm z_i(t)+\bm z_{\bm b,\ell}(t),\\
L(t)
&:=
L_0
+
\int_0^t
\mathcal T\bigl(\bm B(\tau),\mathcal C_{\mathrm d}\bm B(\tau)\bigr)
\,\mathrm d\tau
+
\int_0^t
M_{\ext}(\tau)\,\mathrm d\tau.
\end{align*}
This defines a map
\[
\mathcal F_{\delta,R}:\mathcal S_{\delta,R}\to
C([0,\delta];X_B)\times C([0,\delta];\R).
\]
We first show that $\mathcal F_{\delta,R}$ maps $\mathcal S_{\delta,R}$ into
itself for sufficiently small $\delta>0$. If $(\bm b,\ell)\in\mathcal S_{\delta,R}$, then
\[
\|\bm u_{\bm b,\ell}\|_{L^\infty([0,\delta];Y_{\mathrm d})}
\leq
4c_vR^2.
\]
Hence Lemma~\ref{lem:basic_input_output_estimates} yields
\[
\|\bm z_{\bm b,\ell}\|_{C([0,\delta];X_B)}
\leq
4c\,c_v\sqrt{\delta}\,R^2.
\]
Consequently,
\[
\|\bm B\|_{C([0,\delta];X_B)}
\leq
\|\bm B_0\|_{X_B}
+
\|\bm z_i^{T_*}\|_{C([0,T_*];X_B)}
+
4c\,c_v\sqrt{\delta}\,R^2.
\]
Thus, for $\delta>0$ sufficiently small,
\[
\|\bm B\|_{C([0,\delta];X_B)}
\leq
2R.
\]
Moreover,
\[
\|L-L_0\|_{C([0,\delta])}
\leq
\int_0^\delta
|\mathcal T(\bm B(\tau),\mathcal C_{\mathrm d}\bm B(\tau))|
\,\mathrm d\tau
+
\|M_{\ext}\|_{L^1([0,\delta])}.
\]
Using the boundedness of $\mathcal T$ and Hölder's inequality, we obtain
\[
\int_0^\delta
|\mathcal T(\bm B(\tau),\mathcal C_{\mathrm d}\bm B(\tau))|
\,\mathrm d\tau
\leq
c_{\mathcal T}
\|\bm B\|_{C([0,\delta];X_B)}
\sqrt{\delta}\,
\|\mathcal C_{\mathrm d}\bm B\|_{L^2([0,\delta];Y_{\mathrm d})}.
\]
The inhomogeneous MQS estimate applied to the three terms
$T(\cdot)\bm B_0$, $\bm z_i$, and $\bm z_{\bm b,\ell}$ gives a bound for
$\|\mathcal C_{\mathrm d}\bm B\|_{L^2([0,\delta];Y_{\mathrm d})}$ depending only
on $R$, $\|\bm B_0\|_{X_B}$, and $\|i\|_{L^2([0,T_*];\R^m)}$, for
$0<\delta\leq T_*$. Thus the last display tends to zero as
$\delta\downarrow0$. Also
$\|M_{\ext}\|_{L^1([0,\delta])}\to0$ as $\delta\downarrow0$. Hence, decreasing
$\delta>0$ if necessary,
$\|L\|_{C([0,\delta])}\leq 2R$. 
It follows that
$\|(\bm B,L)\|_\delta\leq2R$, and hence
$\mathcal F_{\delta,R}$ maps $\mathcal S_{\delta,R}$ into itself.

We next show that $\mathcal F_{\delta,R}$ is a contraction for sufficiently
small $\delta>0$. Let
$(\bm b_1,\ell_1),(\bm b_2,\ell_2)\in\mathcal S_{\delta,R}$
and set
\[
d_\delta
:=
\|(\bm b_1-\bm b_2,\ell_1-\ell_2)\|_\delta.
\]
Since both pairs belong to $\mathcal S_{\delta,R}$, there is a constant
$C_R>0$, independent of $\delta$, such that
\[
\|
\bm u_{\bm b_1,\ell_1}
-
\bm u_{\bm b_2,\ell_2}
\|_{L^\infty([0,\delta];Y_{\mathrm d})}
\leq
C_R d_\delta.
\]
Indeed,
\[
\begin{split}
\bm u_{\bm b_1,\ell_1}
-
\bm u_{\bm b_2,\ell_2}
&=
\bm v_{\ell_1}\times(\bm b_1-\bm b_2)
+
\bm v_{\ell_1-\ell_2}\times\bm b_2,
\end{split}
\]
and the asserted estimate follows from
\eqref{eq:disk_velocity_field} and the bounds defining
$\mathcal S_{\delta,R}$.

Let
\[
\bm z_{12}(t)
:=
\int_0^t
T_{-1/2}(t-\tau)\mathcal B_{\mathrm d}
\bigl(
\bm u_{\bm b_1,\ell_1}(\tau)
-
\bm u_{\bm b_2,\ell_2}(\tau)
\bigr)
\,\mathrm d\tau.
\]
For the corresponding magnetic components
$\bm B_1,\bm B_2$, we have
$\bm B_1-\bm B_2=\bm z_{12}$.
Lemma~\ref{lem:basic_input_output_estimates} therefore gives
\[
\|\bm B_1-\bm B_2\|_{C([0,\delta];X_B)}
+
\|
\mathcal C_{\mathrm d}(\bm B_1-\bm B_2)
\|_{L^2([0,\delta];Y_{\mathrm d})}
\leq
C_R\sqrt{\delta}\,d_\delta.
\]

Moreover, the self-mapping estimates and
Lemma~\ref{lem:basic_input_output_estimates} imply that there is a constant
$C_{R,T_*}>0$, independent of $\delta$, such that
\[
\|\bm B_j\|_{C([0,\delta];X_B)}
+
\|\mathcal C_{\mathrm d}\bm B_j\|_{L^2([0,\delta];Y_{\mathrm d})}
\leq
C_{R,T_*},
\qquad j=1,2.
\]
Using the bilinearity and boundedness of $\mathcal T$, we obtain
\[
\begin{split}
|L_1(t)-L_2(t)|
&\leq
\int_0^t
\left|
\mathcal T
\bigl(
\bm B_1(\tau),
\mathcal C_{\mathrm d}\bm B_1(\tau)
\bigr)
-
\mathcal T
\bigl(
\bm B_2(\tau),
\mathcal C_{\mathrm d}\bm B_2(\tau)
\bigr)
\right|
\,\mathrm d\tau
\\
&\leq
c_{\mathcal T}
\|\bm B_1-\bm B_2\|_{C([0,\delta];X_B)}
\|
\mathcal C_{\mathrm d}\bm B_1
\|_{L^1([0,\delta];Y_{\mathrm d})}
+
c_{\mathcal T}
\|\bm B_2\|_{C([0,\delta];X_B)}
\|
\mathcal C_{\mathrm d}(\bm B_1-\bm B_2)
\|_{L^1([0,\delta];Y_{\mathrm d})}.
\end{split}
\]
By Hölder's inequality and the preceding estimates,
$\|L_1-L_2\|_{C([0,\delta])}
\leq
C_{R,T_*}\delta\,d_\delta$.
Consequently,
\[
\|
\mathcal F_{\delta,R}(\bm b_1,\ell_1)
-
\mathcal F_{\delta,R}(\bm b_2,\ell_2)
\|_\delta
\leq
C_{R,T_*}
\bigl(
\sqrt{\delta}+\delta
\bigr)
d_\delta.
\]
Choosing $\delta>0$ sufficiently small, the constant on the right-hand side
is strictly smaller than one. Hence
$\mathcal F_{\delta,R}$ is a strict contraction on
$\mathcal S_{\delta,R}$, and Banach's fixed point theorem gives a unique
fixed point in $\mathcal S_{\delta,R}$.

This fixed point is also locally unique in the full solution class. Indeed,
any solution with the same initial value belongs, by continuity, to
$\mathcal S_{\eta,R}$ after possibly decreasing its interval of definition
to some $\eta>0$. It must therefore coincide with the fixed point on
$[0,\eta]$. Repeating this argument proves uniqueness on every common
interval of existence.

The local construction can be restarted at every time
$t_0$ at which the solution is defined, with initial value
$(\bm B(t_0),L(t_0))$. By the local uniqueness just proved, the resulting
local solutions agree on overlapping intervals. They therefore concatenate
to a unique maximal solution on an interval
$[0,t_{\max})$, where $t_{\max}\in(0,\infty]$.

Suppose that $t_{\max}<\infty$ and that
\[
\sup_{0\leq t<t_{\max}}
\bigl(
\|\bm B(t)\|_{X_B}+|L(t)|
\bigr)
<\infty.
\]
Choose $R_*>0$ such that
$\|\bm B(t)\|_{X_B}+|L(t)|
\leq R_*$ for all $t\in[0,t_{\max})$.
Since
$i\in L^2([0,t_{\max}+1];\R^m)$,
$M_{\ext}\in L^1([0,t_{\max}+1])$,
the absolute continuity of the Lebesgue integral implies that
$\|i\|_{L^2([t_0,t_0+\delta];\R^m)}$ and
$\|M_{\ext}\|_{L^1([t_0,t_0+\delta])}$
can be made uniformly small for
$t_0\in[0,t_{\max}]$ by choosing $\delta>0$ sufficiently small.

Applying the above fixed-point construction after time translation therefore
gives a number $\delta_*>0$, depending only on $R_*$, the fixed material and
geometric constants, and the input norms on $[0,t_{\max}+1]$, but not on the
starting time $t_0$, such that the solution can be continued from every
$t_0<t_{\max}$ to $[t_0,t_0+\delta_*]$.

Choosing $t_0<t_{\max}$ sufficiently close to $t_{\max}$ gives
$t_0+\delta_*>t_{\max}$ and hence extends the solution beyond
$t_{\max}$, contradicting maximality. Therefore, whenever $t_{\max}<\infty$,
\[
\limsup_{t\nearrow t_{\max}}
\bigl(
\|\bm B(t)\|_{X_B}+|L(t)|
\bigr)
=
\infty.
\]
\end{proof}
\begin{theorem}[Global solvability and uniqueness of weak solutions]
\label{thm:weak_solvability}
Assume Hypotheses~\ref{hyp:geometry}--\ref{hyp:disk}. Let $T>0$ and let
$(\bm B_0,L_0)\in X_B\times\R$. Then the integrated
variation-of-constants system \eqref{eq:coupled_voc_fixed_point} has a unique
solution
\[
(\bm B,L)\in
\bigl(L^2([0,T];V_B)\cap C([0,T];X_B)\bigr)
\times C([0,T]).
\]
Moreover, the reconstructed quadruple
$(\mathfrak B,L,\bm E,\bm J_{\mathrm d})$ defined by
\eqref{eq:weak_solution_reconstruction} is the unique weak solution on
$[0,T]$ in the sense of Definition~\ref{def:weak_solution}.

The weak solution satisfies
\begin{align*}
\bm B&\in L^2([0,T];V_B)\cap H^1([0,T];V_B')\cap C([0,T];X_B),&
L&\in W^{1,1}([0,T]),
\\
\bm E&\in L^2([0,T];L^2(\Omega_{\mathrm c};\R^3)),
&\bm J_{\mathrm d}&\in L^2([0,T];L^2(\Omega_{\mathrm d};\R^3)).
\end{align*}
Furthermore, there is a constant $C>0$, depending only on the material and
geometric constants and on $I_{\mathrm d}$, such that
\[
\sup_{t\in[0,T]}
\left(
\|\bm B(t)\|_{X_B}^2+|L(t)|^2
\right)
\leq
C
\left(
\|\bm B_0\|_{X_B}^2+|L_0|^2
+
\|i\|_{L^2([0,T];\R^m)}^2
+
\|M_{\ext}\|_{L^1([0,T])}^2
\right).
\]
Finally, the weak solution satisfies the energy balance stated in
Proposition~\ref{prop:energy_balance_weak_solution}.
\end{theorem}
\begin{proof}
Let $[0,t_{\max})$ be the maximal interval of existence given by
Proposition~\ref{prop:local_solvability_integral_system}. We first derive an
a priori bound for $\|\bm B(t)\|_{X_B}+|L(t)|$ on compact subintervals of
$[0,t_{\max})$.

By construction of the integral solution and by
Lemma~\ref{lem:basic_input_output_estimates}, we have
\[
\bm B\in L^2_{\mathrm{loc}}([0,t_{\max});V_B)
\cap
H^1_{\mathrm{loc}}([0,t_{\max});V_B')
\cap
C([0,t_{\max});X_B),
\]
and $L\in W^{1,1}_{\mathrm{loc}}([0,t_{\max}))$. Hence the energy identity
associated with the form equation, as used in
Theorem~\ref{thm:inhomogeneous_mqs_equation}, gives, for almost every
$t\in(0,t_{\max})$,
\[
\tfrac12\tfrac{\mathrm d}{\mathrm dt}\|\bm B(t)\|_{X_B}^2
+
\mathfrak a(\bm B(t),\bm B(t))
=
\left\langle
\mathcal B_{\mathrm{em}}i(t),\bm B(t)
\right\rangle_{V_B',V_B}
+
\left\langle
\mathcal B_{\mathrm d}(\bm v_{L(t)}\times\bm B(t)),\bm B(t)
\right\rangle_{V_B',V_B}.
\]
Moreover,
\[
\frac{\mathrm d}{\mathrm dt}\frac{1}{2I_{\mathrm d}}L(t)^2
=
I_{\mathrm d}^{-1}L(t)
\mathcal T\bigl(\bm B(t),\mathcal C_{\mathrm d}\bm B(t)\bigr)
+
I_{\mathrm d}^{-1}L(t)M_{\ext}(t).
\]

The Lorentz coupling terms cancel. Indeed, by the definitions of
$\mathcal B_{\mathrm d}$, $\mathcal C_{\mathrm d}$, $\mathcal T$, and
$\bm v_L$, together with the scalar triple product identity,
\[
\left\langle
\mathcal B_{\mathrm d}(\bm v_{L(t)}\times\bm B(t)),\bm B(t)
\right\rangle_{V_B',V_B}
+
I_{\mathrm d}^{-1}L(t)
\mathcal T\bigl(\bm B(t),\mathcal C_{\mathrm d}\bm B(t)\bigr)
=0.
\]
Thus, for
\[
\mathcal E_{\mathrm{abs}}(t)
:=
\tfrac12\|\bm B(t)\|_{X_B}^2
+
\tfrac{1}{2I_{\mathrm d}}L(t)^2,
\]
we obtain
\[
\tfrac{\mathrm d}{\mathrm dt}\mathcal E_{\mathrm{abs}}(t)
+
\mathfrak a(\bm B(t),\bm B(t))
=
\left\langle
\mathcal B_{\mathrm{em}}i(t),\bm B(t)
\right\rangle_{V_B',V_B}
+
I_{\mathrm d}^{-1}L(t)M_{\ext}(t).
\]
Writing $K\bm B(t)=(\bm q(t),\bm h(t))$, the definition of
$\mathcal B_{\mathrm{em}}$ and Young's inequality give
\[
\left|
\left\langle
\mathcal B_{\mathrm{em}}i(t),\bm B(t)
\right\rangle_{V_B',V_B}
\right|
\leq
c_\chi\|i(t)\|_2\|\bm q(t)\|_{L^2(\Omega_{\mathrm c};\R^3)}
\leq
\tfrac12\mathfrak a(\bm B(t),\bm B(t))
+
\frac{c_\chi^2}{2}\|i(t)\|_2^2,
\]
with some uniform $c_\chi >0$. Integrating the energy identity and using the preceding estimate gives, for
every $t<t_{\max}$,
\[
\mathcal E_{\mathrm{abs}}(t)
+
\frac12
\int_0^t
\mathfrak a(\bm B(\tau),\bm B(\tau))
\,\mathrm d\tau
\leq
\mathcal E_{\mathrm{abs}}(0)
+
\frac{c_\chi^2}{2}
\|i\|_{L^2([0,t];\R^m)}^2
+
\int_0^t
I_{\mathrm d}^{-1}|L(\tau)|\,|M_{\ext}(\tau)|
\,\mathrm d\tau.
\]
Since
\[
I_{\mathrm d}^{-1}|L(\tau)|
\leq
\sqrt{\frac{2}{I_{\mathrm d}}}\,
\mathcal E_{\mathrm{abs}}(\tau)^{1/2},
\]
we obtain
\[
\mathcal E_{\mathrm{abs}}(t)
\leq
a(t)
+
\sqrt{\frac{2}{I_{\mathrm d}}}
\int_0^t
\mathcal E_{\mathrm{abs}}(\tau)^{1/2}
|M_{\ext}(\tau)|
\,\mathrm d\tau,
\]
where
\[
a(t)
:=
\mathcal E_{\mathrm{abs}}(0)
+
\frac{c_\chi^2}{2}
\|i\|_{L^2([0,t];\R^m)}^2.
\]
Set
\[
Y(t)
:=
\sup_{0\leq s\leq t}
\mathcal E_{\mathrm{abs}}(s)^{1/2}.
\]
Then
\[
Y(t)^2
\leq
a(t)
+
\sqrt{\frac{2}{I_{\mathrm d}}}\,
Y(t)
\|M_{\ext}\|_{L^1([0,t])}.
\]
Solving the quadratic inequality $y^2-by-a\leq0$ and using
$\sqrt{b^2+4a}\leq b+2\sqrt a$, we obtain
\[
y\leq \frac{b+\sqrt{b^2+4a}}{2}\leq \sqrt a+b.
\]
Consequently,
\[
Y(t)
\leq
a(t)^{1/2}
+
\sqrt{\frac{2}{I_{\mathrm d}}}\,
\|M_{\ext}\|_{L^1([0,t])},
\]
and therefore
\[
\sup_{0\leq s\leq t}
\mathcal E_{\mathrm{abs}}(s)
\leq
2\mathcal E_{\mathrm{abs}}(0)
+
c_\chi^2
\|i\|_{L^2([0,t];\R^m)}^2
+
\frac{4}{I_{\mathrm d}}
\|M_{\ext}\|_{L^1([0,t])}^2.
\]
Since $\mathcal E_{\mathrm{abs}}$ defines a norm equivalent to
$(\bm B,L)\mapsto
\|\bm B\|_{X_B}^2+|L|^2$
on $X_B\times\R$, this proves the asserted a priori estimate.

Suppose now that $t_{\max}<\infty$. Since
$i\in L^2([0,t_{\max}];\R^m)$ and
$M_{\ext}\in L^1([0,t_{\max}])$, the preceding estimate implies
\[
\sup_{0\leq t<t_{\max}}
\bigl(
\|\bm B(t)\|_{X_B}+|L(t)|
\bigr)
<\infty.
\]
This contradicts the blow-up alternative in
Proposition~\ref{prop:local_solvability_integral_system}. Hence
$t_{\max}=\infty$.
In particular, the variation-of-constants system has a unique solution on
$[0,T]$ for every $T>0$. 

By Proposition~\ref{prop:weak_solution_voc_equivalence}, the reconstruction
\eqref{eq:weak_solution_reconstruction} is a weak solution in the sense of
Definition~\ref{def:weak_solution}. The asserted regularity follows from
Theorem~\ref{thm:inhomogeneous_mqs_equation} and the scalar integral equation
for $L$. The energy balance follows from
Proposition~\ref{prop:energy_balance_weak_solution}.

Conversely, every weak solution gives rise, by
Proposition~\ref{prop:weak_solution_voc_equivalence}, to a solution of the
variation-of-constants system \eqref{eq:coupled_voc_fixed_point}. Since this
system has a unique solution in the stated class, the pair $(\bm B,L)$ is
unique. By Lemma~\ref{lem:finite_dissipation_representatives_unique}, the
finite-dissipation representative is uniquely determined by $\bm B$, and
\eqref{eq:weak_solution_reconstruction} uniquely determines $\bm E$ and
$\bm J_{\mathrm d}$. Hence the weak solution is unique.
\end{proof}

\section{Possible model extensions}
\label{sec:model_extensions}

We briefly indicate two natural extensions of the model. The purpose of this
section is not to provide a complete solvability theory, but to record the
corresponding model modifications and to point out which parts of the above
analysis may be expected to carry over.

\subsection{Coupling with the full Maxwell system}
\label{subsec:extension_full_maxwell}

A first possible extension is to replace the magneto-quasistatic
approximation by the full Maxwell system. Then the electric field is no
longer an algebraic variable, but becomes part of the dynamic state. Let the
electric permittivity
$\varepsilon:\Omega\to\R^{3\times3}$ be pointwise symmetric and uniformly
positive definite, with
$\varepsilon,\varepsilon^{-1}\in L^\infty(\Omega;\R^{3\times3})$.
The formal model is obtained from \eqref{eq:model} by replacing the algebraic
Ampere constraint with Maxwell's evolution equation,
\begin{align*}
\tfrac{\mathrm{d}}{\mathrm{d}t} \bm B(\xi,t)
  &=
  -\curl\bm E(\xi,t),
  && \xi\in\Omega,
  \\
\varepsilon(\xi)\tfrac{\mathrm{d}}{\mathrm{d}t}\bm E(\xi,t)
  &=
  \curl\bigl(\nu\bm B\bigr)(\xi,t)-\bm J(\xi,t),
  && \xi\in\Omega,
  \\
\dot L(t)
  &=
  \bm e_x^\top
  \int_{\Omega_{\mathrm d}}
  \bm r(\xi)\times
  \bigl(\bm J_{\mathrm d}(\xi,t)\times\bm B(\xi,t)\bigr)
  \,\mathrm d\xi
  +
  M_{\ext}(t),
  &&
  \\
\bm J(\xi,t)
  &=
  \begin{cases}
  \sigma_{\mathrm{em}}(\xi)\bm E(\xi,t)
  +
  \bm J_{\mathrm{em}}(\xi,t),
  & \xi\in\Omega_{\mathrm{em}},
  \\
  \bm J_{\mathrm d}(\xi,t),
  & \xi\in\Omega_{\mathrm d},
  \\
  0,
  & \xi\in\Omega_0,
  \end{cases}
  \\
\bm J_{\mathrm d}(\xi,t)
  &=
  \sigma_{\mathrm d}(\xi)
  \left(
  \bm E(\xi,t)
  +
  \bm v(\xi,t)\times\bm B(\xi,t)
  \right),
  && \xi\in\Omega_{\mathrm d},
  \\
\bm E(\xi,t)\times\bm n_o(\xi)
  &=
  -\zeta(\xi)\,
  \bm n_o(\xi)\times
  \bigl(\nu(\xi)\bm B(\xi,t)\times\bm n_o(\xi)\bigr),
  && \xi\in\partial\Omega.
\end{align*}
where
\begin{equation*}
\bm r(\xi):=\xi-c,\qquad
\bm v(\xi,t):=I_{\mathrm d}^{-1}L(t)(\bm e_x\times\bm r(\xi)),
\qquad
\bm J_{\mathrm{em}}(\xi,t):=\chi(\xi)i(t).
\end{equation*}
The dynamic state of this model is $(\bm E,\bm B,L)$.

Formally, the corresponding energy is
\[
\mathcal E_{\mathrm{full}}(t)
:=
\frac12
\int_\Omega
(\varepsilon\bm E(t))^\top\bm E(t)
\,\mathrm d\xi
+
\frac12
\int_\Omega
(\nu\bm B(t))^\top\bm B(t)
\,\mathrm d\xi
+
\frac{1}{2I_{\mathrm d}}L(t)^2.
\]
The Lorentz coupling terms cancel as in
Proposition~\ref{prop:energy_balance_weak_solution}. Indeed, in the disk
\[
\bm E
=
\sigma_{\mathrm d}^{-1}\bm J_{\mathrm d}
-
\bm v\times\bm B,
\]
so that the electromagnetic power contribution contains
\[
-\int_{\Omega_{\mathrm d}}
\bm J_{\mathrm d}^\top\bm E
\,\mathrm d\xi
=
-\int_{\Omega_{\mathrm d}}
\bm J_{\mathrm d}^\top\sigma_{\mathrm d}^{-1}\bm J_{\mathrm d}
\,\mathrm d\xi
+
\int_{\Omega_{\mathrm d}}
(\bm v\times\bm B)^\top\bm J_{\mathrm d}
\,\mathrm d\xi,
\]
and the last term cancels with the mechanical power generated by the Lorentz
torque.

For the solvability analysis, the full Maxwell system leads to a different
linear realization than the magneto-quasistatic model. In the undamped case,
the Maxwell operator is naturally described by a strongly continuous group on
the electromagnetic energy space \cite{StaffansWeiss2013}. With conductivity and impedance boundary
damping, one obtains a dissipative Maxwell semigroup. In contrast to the
magneto-quasistatic formulation above, the motional eddy current term enters
additively in the electric field equation:
\[
\sigma_{\mathrm d}(\bm v_L\times\bm B)
\in L^2(\Omega_{\mathrm d};\R^3).
\]
Thus this nonlinear term already takes values in the state space of the
Maxwell system, and no extrapolation space is needed for this input. Moreover,
the Lorentz torque depends continuously on the electromagnetic state and on
$L$. Hence a Picard iteration for the triple $(\bm E,\bm B,L)$ can be expected
to yield local solvability, while the formal energy balance provides the
natural a priori estimate for continuation.
\subsection{Non-axisymmetric rotating disks}
\label{subsec:extension_nonaxisymmetric_disk}

A second possible extension concerns disk geometries which are not invariant
under rotations about the axis. In the present paper the disk domain is fixed
in the laboratory frame. This is consistent with a rotationally symmetric
disk, since rotating the disk does not change the occupied conducting region.
For a non-axisymmetric disk, for instance a slotted or perforated disk, the
conducting region and the material coefficients depend on the angular
position of the disk. Hence the angular position has to be included as an
additional state variable.

Let $\theta(t)\in\R$ denote the angular position of the disk. The
angular kinematics are
\[
\dot\theta(t)=I_{\mathrm d}^{-1}L(t).
\]
The rigid-body velocity field is still given by
\[
\bm v(\xi,t)
=
I_{\mathrm d}^{-1}L(t)(\bm e_x\times\bm r(\xi)).
\]
In contrast to the rotationally symmetric case, however, the disk
conductivity is now angle-dependent. We write
\[
\sigma_{\mathrm d}=\sigma_{\mathrm d}(\xi,\theta),
\]
where the second argument is $2\pi$-periodic,
\[
\sigma_{\mathrm d}(\xi,\theta+2\pi)
=
\sigma_{\mathrm d}(\xi,\theta).
\]
For each fixed $\theta$, its support describes the
conducting region occupied by the disk at angular position $\theta$.

With this notation, the motional electric field in the disk is
$\bm v(\xi,t)\times\bm B(\xi,t)$,
and total disk current density becomes
\[
\bm J_{\mathrm d}(\xi,t)
=
\sigma_{\mathrm d}(\xi,\theta(t))
\left(
\bm E(\xi,t)
+
\bm v(\xi,t)\times\bm B(\xi,t)
\right).
\]
The Lorentz force density acting on the disk is therefore
\[
\bm f_{\mathrm L}(\xi,t)
=
\bm J_{\mathrm d}(\xi,t)\times\bm B(\xi,t),
\]
and the corresponding torque about the rotation axis is
\[
\bm e_x^\top
\int_{\Omega}
\bm r(\xi)\times
\bigl(
\bm J_{\mathrm d}(\xi,t)\times\bm B(\xi,t)
\bigr)
\,\mathrm d\xi.
\]
Here the integral may be written over all of $\Omega$, since
$\sigma_{\mathrm d}(\cdot,\theta(t))$ vanishes outside the disk at angular
position $\theta(t)$.

The corresponding magneto-quasistatic model takes the form
\begin{align*}
\tfrac{\mathrm{d}}{\mathrm{d}t} \bm B(\xi,t)
  &=
  -\curl\bm E(\xi,t),
  && \xi\in\Omega,
  \\
\dot L(t)
  &=
  \bm e_x^\top
  \int_{\Omega_{\mathrm d}(\theta(t))}
  \bm r(\xi)\times
  \bigl(\bm J_{\mathrm d}(\xi,t)\times\bm B(\xi,t)\bigr)
  \,\mathrm d\xi\\&\qquad
  +
  M_{\ext}(t),
  &&
  \\
\dot\theta(t)
  &=
  I_{\mathrm d}^{-1}L(t),
  &&
  \\
\bm J_{\mathrm d}(\xi,t)
  &=
  \sigma_{\mathrm d}(\xi,\theta(t))
  \left(
  \bm E(\xi,t)
  +
  \bm v(\xi,t)\times\bm B(\xi,t)
  \right),
  && \xi\in\Omega_{\mathrm d}(\theta(t)),
  \\
\curl\bigl(\nu\bm B\bigr)(\xi,t)
  &=
  \begin{cases}
  \sigma_{\mathrm{em}}(\xi)\bm E(\xi,t)
  +
  \bm J_{\mathrm{em}}(\xi,t),
  & \xi\in\Omega_{\mathrm{em}},
  \\
  \bm J_{\mathrm d}(\xi,t),
  & \xi\in\Omega_{\mathrm d}(\theta(t)),
  \\
  0,
  & \xi\in\Omega_0(\theta(t)),
  \end{cases}
  \\
\bm E(\xi,t)\times\bm n_o(\xi)
  &=
  -\zeta(\xi)\,
  \bm n_o(\xi)\times
  \bigl(\nu(\xi)\bm B(\xi,t)\times\bm n_o(\xi)\bigr),
  && \xi\in\partial\Omega.
\end{align*}
where
\begin{equation*}
\begin{gathered}
\Omega_{\mathrm d}(\theta)
:=
c+R_\theta
\bigl(\Omega_{\mathrm d}^{\mathrm{ref}}-c\bigr),
\qquad
\Omega_0(\theta)
:=
\Omega\setminus
\overline{\Omega_{\mathrm{em}}\cup\Omega_{\mathrm d}(\theta)},\\
\bm r(\xi):=\xi-c,\qquad
\bm v(\xi,t):=I_{\mathrm d}^{-1}L(t)(\bm e_x\times\bm r(\xi)),
\qquad
\bm J_{\mathrm{em}}(\xi,t):=\chi(\xi)i(t).
\end{gathered}
\end{equation*}
Here $R_\theta\in\R^{3\times3}$ is the matrix describing rotation by
the angle $\theta$ about the axis $\{c+s\bm e_x:s\in\R\}$.

The boundary condition on $\partial\Omega$ can be kept in the same
Silver--Müller form as in \eqref{eq:model_boundary}.

From the analytical point of view, the additional state variable $\theta$
introduces a state-dependent geometry. In the rotationally symmetric case,
the disk domain, the conductivity, the disk current map, and the torque map
are fixed. In the non-axisymmetric case, these objects depend on the angular
position of the disk. Thus the magnetic equation is no longer governed by a
single autonomous magneto-quasistatic realization on a fixed disk geometry.

One possible way to formulate the problem on fixed spaces is to pull the disk
equations back to a reference configuration
$\Omega_{\mathrm d}^{\mathrm{ref}}$. Then the geometry is fixed, while the
conductivity and the coupling terms become explicitly $\theta$-dependent.
A well-posedness analysis could then be attempted by a fixed-point argument
for the triple $(\bm B,L,\theta)$, provided that the angle-dependent
coefficients depend sufficiently regularly on $\theta$ and satisfy uniform
boundedness and positivity assumptions. The Lorentz energy cancellation is
expected to persist, since it is based on the rigid-body velocity field and
the scalar triple product identity rather than on rotational symmetry of the
disk.

The same modeling idea can also be combined with the full Maxwell extension
described in Subsection~\ref{subsec:extension_full_maxwell}. In that case the
state becomes $(\bm E,\bm B,L,\theta)$, and the angle-dependent disk current
enters Ampere's law as an additional nonlinear source term.

\section{Conclusion}
\label{sec:conclusion}

The finite-dissipation formulation provides a convenient framework for the
analysis of the eddy current brake. It separates the closed linear
magneto-quasistatic realization from the nonlinear electromechanical coupling
and makes the energy cancellation mechanism explicit. This yields global
existence and uniqueness of weak solutions under low regularity assumptions on
the inputs.

The approach also leaves room for further model refinements. In particular,
absorbing boundary conditions are incorporated without changing the basic
argument, while extensions to full Maxwell dynamics or non-axisymmetric disks
would require additional analytical work.

% To print the credit authorship contribution details
\printcredits

%% Loading bibliography style file
%\bibliographystyle{model1-num-names}
\bibliographystyle{cas-model2-names}

% Loading bibliography database
\bibliography{eddy_literature}

\end{document}